\documentclass[10pt,twoside]{amsart}
\usepackage[english]{babel}
\usepackage[all]{xy}
\usepackage{amssymb,amsmath,url,xr}
\usepackage{bbm}
\usepackage{bm}
\usepackage[mathscr]{euscript}
\usepackage{mathrsfs}
\usepackage[a4paper,top=2.7cm, bottom=2.6cm, left=2.6cm, right=2.6cm]{geometry}
\usepackage[colorlinks=true,pagebackref=true]{hyperref} 

\title{The homotopy Leray spectral sequence}

\author{Aravind Asok, Fr\'ed\'eric D\'eglise, Jan Nagel}

\date{December 2018}

\newtheorem{thm}{Theorem}[subsection]
\newtheorem{prop}[thm]{Proposition}
\newtheorem{lm}[thm]{Lemma}
\newtheorem{cor}[thm]{Corollary}
\newtheorem{thmi}{Theorem}
\newtheorem{propi}[thmi]{Proposition}

\theoremstyle{remark}
\newtheorem{rem}[thm]{Remark}

\newtheorem{ex}[thm]{Example}

\theoremstyle{definition}
\newtheorem{df}[thm]{Definition}
\newtheorem{num}[thm]{}

\numberwithin{equation}{thm}

\numberwithin{equation}{thm}

\DeclareMathOperator{\trd}{trd}
\DeclareMathOperator{\rk}{rk}

\newcommand{\hrt}[1]{#1^\heartsuit}

\DeclareMathOperator{\Th}{Th}

\DeclareMathOperator{\uK}{\underline K}

\DeclareMathOperator{\Mbm}{M^{BM}}
\newcommand{\dtw}[1]{\langle#1\rangle}
\newcommand{\gtw}[1]{\{#1\}}

\newcommand{\pplim}[1] {\underset{#1}{"\!\varprojlim\!"}}

\newcommand{\mH}{\mathcal H} 

\newcommand{\T}{\mathscr T} 

\DeclareMathOperator{\pts}{Pts}

\newcommand{\cF}{\mathcal F}

\newcommand{\smod}[1]{#1\!-\!mod}
\newcommand{\MGL}{\mathbf{MGL}}
\newcommand{\KGL}{\mathbf{KGL}}

\newcommand{\HH}{\mathbf H}
\newcommand{\MGLmod}{\smod{\MGL}}

\newcommand{\ab}{\mathscr Ab}
\newcommand{\SH}{\mathrm{SH}}

\newcommand{\eff}{\mathit{eff}}

\newcommand{\DM}{\mathrm{DM}}

\DeclareMathOperator{\Hom}{Hom}

\DeclareMathOperator{\uHom}{\underline{Hom}} 

\DeclareMathOperator{\spec}{Spec}

\newcommand{\ilim} { \varinjlim }

\newcommand{\pilim}[1] { \underset{#1}{"\varinjlim\!\!"} }

\DeclareMathOperator{\codim}{codim}

\newcommand{\ZZ} {\mathbf Z}
\newcommand{\QQ} {\mathbf Q}

\newcommand{\CC} {\mathbf C}
\renewcommand{\AA} {\mathbf A}

\newcommand{\GG} {\mathbf{G}_m}

\newcommand{\E}{\mathbb E}
\newcommand{\F}{\mathbb F}
\newcommand{\G}{\mathbb G}

\newcommand{\un}{\mathbbm 1} 
\newcommand{\pur}{\mathfrak p} 

\newcommand{\flag}{\mathcal F}

\newcommand{\A}{\mathscr A}

\newcommand{\Addresses}{{
 \bigskip
 \footnotesize

 A.~Asok, Department of Mathematics, University of Southern California, 3620 S. Vermont Ave.,
  Los Angeles, CA 90089-2532, United States; \textit{E-mail address:} \url{asok@usc.edu}

 \medskip

 F.~D{\'e}glise, IMB - UMR5584, 9 avenue Alain Savary 21000 Dijon Cedex; France \textit{E-mail address:} \url{frederic.deglise@u-bourgogne.fr}

 \medskip

 J.~Nagel, IMB - UMR5584, 9 avenue Alain Savary 21000 Dijon Cedex; France \textit{E-mail address:} \url{Johannes.Nagel@u-bourgogne.fr}

}}

\newcommand{\aone}{\mathbb A^1}

\newcommand{\Spec}[1]{\operatorname{Spec}(#1)}

\begin{document}

\begin{abstract}
In this work, we build a spectral sequence in motivic homotopy that is analogous to both the Serre spectral sequence in algebraic topology and the Leray spectral sequence in algebraic geometry.  Here, we focus on laying the foundations necessary to build the spectral sequence and give a convenient description of its $E_2$-page.  Our description of the $E_2$-page is in terms of homology of the local system of fibers, which is given using a theory similar to Rost's cycle modules.  We close by providing some sample applications of the spectral sequence and some hints at future work.
\end{abstract}

\maketitle
\tableofcontents

\section{Introduction}

The goal of this paper is to study an algebro-geometric version of the Leray--Serre spectral sequence for generalized cohomology theories.  To explain the setup, suppose one is given a topological space $X$, a sheaf $\cF$ of abelian groups on $X$ and a filtration
\[
X_{\bullet} = X_0\subset\ldots\subset X_p\subset X_{p+1}\subset\ldots\subset X_n =X
\]
by closed subsets.  One naturally associates an exact couple with the preceding data and obtains a spectral sequence of the form:
\begin{eqnarray*}
E_1^{p,q}(X_{\bullet},\cF) & = & H^{p+q}(X_p\setminus X_{p-1},\cF)\Rightarrow H^{p+q}(X,\cF).
\end{eqnarray*}
If $X_{\bullet}$ is cellular with respect to $\cF$, i.e., if $H^i(X_p\setminus X_{p-1},\cF) = 0$ for all $i\ne p$, then there are isomorphisms of the form $H^k(X,\cF)\cong E_2^{k,0}(X_{\bullet},\cF)$.  If $f:X\to B$ is a continuous map of topological spaces and $B_{\bullet}$ is a filtration of $B$ that is cellular with respect to the direct image sheaves $R^qf_*\cF$ for all $q$, then the $E_2$--term of Leray spectral sequence
\[
E_2^{p,q} = H^p(B,R^qf_*\cF)\Rightarrow H^{p+q}(X,\cF)
\]
is isomorphic to $E_2^{p,q}(X_{\bullet},\cF)$, where $X_{\bullet}$ is the inverse image of $B_{\bullet}$.

\medskip On the other hand, suppose
\[
F \longrightarrow X \stackrel{f}{\longrightarrow} B
\]
is a Serre fibration of topological spaces, where $B$ has the homotopy type of a (connected) finite CW complex, and $E$ is a (generalized) cohomology theory in the sense of classical stable homotopy theory.  One may consider an associated Atiyah--Hirzebruch spectral sequence (see, e.g., \cite[\S 9.2-9.5]{DavisKirk}): the $E_2$-page of this spectral sequence is given in terms of the ordinary (co)homology of $B$ with coefficients in local systems attached to the $E$-(co)homology of $F$ and converges to the $E$-(co)homology of $X$.  When $E$ is ordinary cohomology and $\mathscr{F}$ is the constant sheaf $\ZZ_X$, the two spectral sequences coincide.  Indeed, Serre showed that the direct image sheaves $R^qf_*\ZZ_X$ are local systems on $B$ depending only on the cohomology of the fiber $F$ and the action of the fundamental group of $B$ on these cohomology groups.


\medskip  To explain our algebro--geometric analog, recall that Arapura showed the Leray spectral sequence is ``motivic" \cite{Arapura}.  In more detail, suppose $f: X \to B$ is a projective morphism of complex quasi--projective algebraic varieties.
After reducing to a situation where $B$ is affine (via the Jouanolou trick), work of Beilinson and Nori \cite{Nori} shows that any constructible sheaf on $B$ can be made cellular with respect to a filtration by closed algebraic subsets. In this situation, Arapura compares the Leray spectral sequence
\[
E_2^{a,b}(f) = H^a(B,R^qf_*\ZZ_X)\Rightarrow H^{a+b}(X,\ZZ_X)
\]
and the spectral sequence associated with the skeletal filtration and uses this to show that the Leray spectral sequence essentially reflects suitably functorial algebro-geometric structure present on cohomology groups (e.g., it lifts to Nori's category of mixed motives).   
A similar result holds for the perverse Leray spectral sequence
\[
E_2^{a,b} = H^a(B,{^p}R^bf_*\ZZ_X)\Rightarrow H^{a+b}(X,\ZZ_X)
\]
that is obtained by replacing the classical truncation functor $\tau$ by the perverse truncation ${^p}\tau$ \cite{dC-M}.

\medskip
In this paper, we consider a variant of this setup.
 Let us now consider a simple case to highlight our considerations.
 Suppose $k$ is a field, $B$ is a finite dimensional irreducible smooth $k$-variety, and $f:  X \to B$ is a smooth morphism of $k$-varieties.  As usual, $f$ may be thought of as an \'etale locally trivial fibration in algebraic geometry.  We write $B^{(n)}$ for the set of codimension $n$ points of $B$.  If $B$ has dimension $d$, we view the collection $B^{(i)}$ as $i$ ranges from $0$ through $d$ as an algebro-geometric analog of the skeletal filtration of $B$.  In this situation, the scheme-theoretic fiber $X_{b}$ for $b \in B^{(i)}$ is then a smooth variety over $b$.

\medskip
There is a comparison theorem between the Leray spectral sequence and the spectral sequence associated to the Bloch-Ogus complex. If $H^*(X,A)$ is a suitable cohomology theory defined using a Grothendieck topology that is finer than the Zariski topology (e.g., \'etale cohomology, or Betti cohomology over $\CC$), one can consider the Leray spectral sequence associated with the morphism of sites $\pi:X_{\rm fine}\to X_{Zar}$. The higher direct image sheaves $R^q\pi_*A$ are the Zariski sheaves $\mH^q$ associated with the presheaves $U\mapsto H^q(\pi^{-1}(U),A)$. Bloch and Ogus show that the Leray spectral sequence $E_2^{p,q} = H^p(X,\mH^q)\Rightarrow H^{p+q}(X,A)$ can be identified from the $E_2$--term onward with the spectral sequence
\[
E_1^{p,q} = \oplus_{x\in X^{(p)}} H^{q-p}(k(x))\Rightarrow N^{\bullet}H^{p+q}(X)
\]
that gives the coniveau filtration on cohomology. (The argument uses Deligne's technique of ``d\'ecalage", see \cite{Paranjape}.)  Furthermore, there are variants of the previous constructions where one  replaces cohomology by Borel--Moore homology and works with the niveau (rather than coniveau) filtration. More generally, one can work with the bivariant theory
\[
H^i(X \buildrel{f}\over\rightarrow Y) = \Hom(Rf_!\QQ_X,\QQ_Y[i])
\]
that simultaneously generalizes both Borel--Moore homology and cohomology.

\medskip
Building on this analysis, by analyzing formal properties of Gersten complexes, Rost developed a notion of ``local coefficient system" on a $k$-scheme $B$, which he called a cycle module.  Given a morphism $f:X\to B$ and a cycle module $M$ on $X$, Rost defined Chow groups with coefficients in $M$, $A_*(X,M)$.  For example, there is a cycle module built out of Milnor K-theory for which the associated Chow groups with coefficients coincide with usual Chow groups.  Finally, Rost constructed a spectral sequence of the form:
\[
E^2_{p,q} = A_p(B,A_q(X,M))\Rightarrow A_{p+q}(X,M);
\]
here $A_q(X,M)$ is a cycle module on $B$ obtained by taking homology of the fibers.

\medskip
Our approach in motivic homotopy theory essentially mixes all the ideas above and combines the classical Leray spectral sequence of a fibration with Rost's spectral sequence.  The jumping off point is the work of the second author exploring the close relationship between the notion of cycle module in the sense of Rost and the heart of a $t$-structure on the motivic stable homotopy category.  We use a relative version of Morel's homotopy $t$-structure, the so-called \emph{perverse homotopy t-structure}, defined initially by Ayoub \cite{Ayoub1} and developed further in \cite{BD1}.  The approach to constructing this $t$-structure in \cite{BD1} is rather flexible and works in great generality: the basic definition requires only the existence of a dimension function $\delta$ on the base scheme; for this reason, the $t$-structure was called the \emph{$\delta$-homotopy $t$-structure} in \emph{op. cit.}.

The heart of the $\delta$-homotopy $t$-structure consists of $\delta$-homotopy modules, or simply {\em homotopy modules}.  These objects are closely related to cycle modules and their homology (with respect to the truncation functors in the $\delta$-homotopy $t$-structure) can be computed by a Gersten-type complex that is formally extremely similar to that written down by Rost (see Definition~\ref{df:Gersten} for more details).  The approach we take via $t$-structures has technical advantages: functoriality properties and multiplicative structure are easy to obtain (see Proposition~\ref{prop:Gersten} and Proposition~\ref{prop:products_Gesten_hlg} for more details).  With these tools in hand, we state a version of our main result here (the notation and terminology in the statement that we have not yet mentioned may be found in the notations and conventions section below).




\begin{thmi}[see \textup{Theorem \ref{thm:Leray_ssp}}]
Fix a base scheme $S$ with a dimension function $\delta$, and suppose $f:  X \to B$ is a separated morphism of $S$-schemes having finite type.  
 Let $\E$ be a motivic spectrum (or a mixed motive) over $S$, and put: $\E^!_X:=f^!\E$.
 Then there exists a convergent spectral sequence of the form:
\[
E^2_{p,q}(f,\E)=A_p^\delta\big(B,H^\delta_{q}(f_*\E^!_X))\big)
\Rightarrow \E_{p+q}(X/S),
\]
abutting to the $\E$-homology of $X$ relative to $S$,
 or rather,
 the bivariant theory of $X/S$ with coefficients in $\E$
 (see below our Notations and conventions).  In fact, $E^2_{p,q}$
  \begin{itemize}
  \item[i)] vanishes outside a range of columns bounded by the maximum and minimum of the dimension function $\delta$ on $B$; and
  \item[ii)] is described as the homology of a Gersten type complex with coefficients in the homotopy module $H^\delta_{q}(f_*\E_X^!))$.
  \end{itemize}
\end{thmi}




The spectral sequence above is obtained by filtering $f_*(\E_X^!)$ by truncations with respect to the $\delta$-homotopy $t$-structure.  It is closely related to the spectral sequence one would obtain by filtering $\E^*(-/S)$ via the pullback of the ($\delta$-)niveau filtration on $B$ along $f$ (for the dimension function $\delta$).  The following result gives a precise formulation.


\begin{propi}[See \textup{Proposition~\ref{prop:comparison_filt}}]
Let ${}^\delta F_p^f\E_n(X/S)$, $p\in \ZZ$, be the abuting filtration of the homotopy Leray spectral sequence.  Define the $\delta$-niveau filtration on $\E_*(X/S)$ relative to $f$ by the following formula:
\[
{}^\delta N_p^f\E_*(X/S)
=\bigcup_{i:Z \rightarrow X, \delta(Z)\leq p}
\mathrm{Im}\big(i_*:\E_*(X \times_B Z/S) \rightarrow \E_*(X/S)\big).
\]
For any pair of integer $(p,n) \in \ZZ^2$, the following relation holds:
\[
{}^\delta F_p^f\E_n(X/S)={}^\delta N_{n-p}^f\E_n(X/S).
\]
\end{propi}

This proposition may be viewed as an analogue of Washnitzer's conjecture (that the coniveau filtration on de Rham cohomology coincides with the filtration arising from the second hypercohomology spectral sequence), which was established by Bloch and Ogus \cite{BO}.


There is also a cohomological form of the homotopy Leray spectral sequence, which is better suited to analysis of product structures.  We


\begin{thmi}[See \textup{Theorem \ref{thm:deltahomotopylerayss}}]
Fix a base scheme $S$ with a dimension function $\delta$, and suppose $f:  X \to B$ is a separated morphism of $S$-schemes having finite type.  If $\E$ is a motivic ring spectrum over $X$, then there is a convergent spectral sequence of the form
\[
E_2^{p,q}(f,\E)=A^p_\delta(B,H^q_\delta(f_*\E_X)) \Rightarrow H^{p+q}(X,\E),
\]
where we write $H^*(X,\E)$ for the groups usually denoted $\E^*(X)$ to emphasize the link with Leray spectral sequence.  As before the $E_2$-page is a) concentrated in a range of columns bounded by the maximum and minimum of the dimension function $\delta$ on $B$, and b) described in terms of (co)homology of the Gersten-style complex associated with homotopy module $H^q_\delta(f_*\E_X)$.  Moreover,
\begin{itemize}
\item[i)] each page $E_r$ has the structure of a differential graded algebra, and
\item[ii)] the induced product structure on the abutment coincides with the $\E$-cohomology product structure.
\end{itemize}
\end{thmi}





\medskip
While the reader may consult the beginning of each section for further discussion of its contents, we close the introduction with a brief overview of the paper.  Section~\ref{sec:homtopytstructure} is devoted to reviewing the basic properties of this homotopy $t$--structure.  Section~\ref{sec:fiberhomology} then studies homotopy modules, Gersten complexes and the associated niveau spectral sequences.  In particular, we describe multiplicative structure on the niveau spectral sequences when considering cohomology theories equipped with suitable multiplicative structure.  Section~\ref{sec:Leray_ssq} is the theoretical heart of the paper.  We build the homotopy Leray spectral sequence (see, e.g., Theorem~\ref{thm:Leray_ssp}) and study an analog of ``locally constant sheaves" in our context.  Finally, Section~\ref{sec:applications} contains some sample applications of these spectral sequences.  In particular, we analyze fibrations with $\AA^1$-contractible fibers, construct Gysin and Wang sequences and, in the motivic case, prove a degeneration result for the homotopy Leray spectral sequence for relative cellular spaces.


\subsubsection*{Acknowledgements}
The authors want to thank Adrien Dubouloz for fruitful discussions about fibrations, and, especially, Fabien Morel for providing the initial impetus to pursue this work.  Aravind Asok was partially supported by National Science Foundation Awards DMS-1254892 and DMS-1802060.  F. D\'eglise and J. Nagel received support from the French "Investissements d’Avenir" program, project ISITE-BFC (contract ANR-lS-IDEX-OOOB).  

\section*{Notations and conventions}
\label{sec:notations}


\noindent \textbf{Geometry}.
All schemes in this paper will be noetherian, finite dimensional and assumed to come equipped with a dimension function, usually denoted $\delta$.  While we fix such a dimension function throughout, and while it may even appear explicitly in various notions we use, we emphasize that the choice is inessential in the sense that the most of the relevant notions do not depend on the fixed choice of dimension function up to a suitable notion of canonical equivalence; see Remark \ref{rem:delta_independence} for a more precise statement.\footnote{There is an exception to this rule. this is when one considers the underlying monoidal structure and its interaction with the $\delta$-homotopy $t$-structure. It will appear once in our applications, when we consider products on the homotopy Leray spectral sequence. We refer the reader to Sections \ref{sec:products} and \ref{sec:coh_version}.}
Anyway the dimension function is used to fix conventions regarding degrees of (co)homology as in \ref{num:homological_delta_ref}.

Fix a base scheme $S$.  An $S$-scheme $X$ will be said to have essentially finite type or, equivalently, $X \to S$ has essentially finite type, if $X$ can be written as a (co)filtered limit a $S$-schemes of finite type with affine \'etale transition morphisms.  By a point of a scheme $X$, we will mean a map $x:\spec(K) \rightarrow X$ where i) $K$ is a field and ii) if $k$ is the residue field of the image of $x$ in $X$, then the extension $K/k$ is finitely generated; equivalently, the morphism $x:\spec(K) \rightarrow X$ has essentially finite type.
We will simply write such a point by $x \in X(K)$.

The dimension function $\delta$ on $S$, is extended to a dimension function on $X$ as follows:
\begin{align*}
\forall x/s \in X/S,
 \delta(x)&:=\delta(s)+\trd(\kappa(x)/\kappa(s)). \\
\delta(X)=\delta_+(X)&:=\max_{x \in X}(\delta(x)). \\
\delta_-(X)&:=\min_{x \in X}(\delta(x)).
\end{align*}
Two examples to keep in mind include:
\begin{enumerate}
\item[(S1)] if $S$ is the spectrum of a field $k$, $\delta=0$, then $\delta(x)=\trd(\kappa(x)/k)$;
\item[(S2)] if $S$ is an excellent regular scheme of dimension less or equal than $3$ and $\delta$ is the Krull dimension.
\end{enumerate}
In both cases, if $X/S$ has finite type, then the integer $\delta(X)$ coincides with the Krull dimension of $X$.

Finally, the following formulas will be used in the paper:
\begin{enumerate}
\item[(D1)] if $S$ is regular connected, $\delta=d-\codim_S$ where $d=\delta(S)$;
\item[(D2)] if the morphism $f:X \rightarrow S$ has essentially finite type, and is lci with cotangent complex $L_f$, for any point $x \in X$, $\delta(x)=\mathrm{rk}\big(L_{f,x}\big)+\delta(s)$.  Moreover, if $X$ and $S$ are irreducible, and $d=\dim(f)$ is the rank of $L_f$, one has:
    \[
    \delta(X)=d+\delta(S).
    \]
\end{enumerate}

\noindent \textbf{Motivic stable homotopy}.
We fix a motivic triangulated category $\T$ in the sense of \cite[Def. 2.4.45]{CD3}, equipped with a motivic adjunction:
$$
\SH \leftrightarrows \T.
$$
In brief, $\T$ consists of the following data: for any scheme $S$, a triangulated closed symmetric monoidal category $\T(S)$; for any morphism of schemes $f$ and any separated morphism of finite type $p$, pairs of adjoint functors $(f^*,f_*)$, $(p_!,p^!)$ satisfying the so-called Grothendieck six functor formalism (see \cite[Th. 2.4.50]{CD3} for a precise statement). Following the terminology from stable homotopy theory objects of $\T(S)$ will be called \emph{$\T$-spectra over $S$}.

We write $\un_S$ for the monoidal unit in $\T(S)$, and $\un_S(1)$ for the Tate twist.  In the sequel, various combinations of Tate twists and shifts naturally arise, and we introduce a separate notation for these twists:
\[
\un_S\gtw 1:=\un_S(1)[1], \ \un_S\dtw 1 :=\un_S(1)[2].
\]
Given an object $\E$ in $\T(S)$, and a morphism $f: X \to S$ (resp. a separated morphism having finite type) we may define $\E$-cohomology and bivariant $\E$-theory by means of the formulas:
\begin{itemize}
\item (\textit{Cohomology}) $\E^{n,i}(X)=\Hom_{\T(X)}(\un_X,f^*\E(i)[n])$.
\item (\textit{Bivariant theory}
 $\E_{n,i}(X/S)=\Hom_{\T(X)}\big(\un_X(i)[n],f^!\E\big)$.
\end{itemize}
This notation is a standard in motivic homotopy theory. When dealing with spectral sequences,
 and homotopy modules, it will both be useful and meaningful to use the following notations:
\begin{itemize}
\item $H^{n}(X,\E)=\E^{n,0}(X)$.
\item $H_{n}(X/S,\E)=\E_{n,0}(X/S)$.
\end{itemize}
We will use the later notations exclusively in Sections \ref{sec:Leray_ssq} and \ref{sec:applications}.

The definition of the bivariant theory attached to $\E$ will be extended to the situation where $X$ has essentially finite type over $S$ (see Paragraph \ref{num:extension}).  The indexing for cohomology used above may also be expressed in terms of the other conventions for twists mentioned above, with notation changed accordingly: one has equalities of the form
\[
\E^{n,i}(X)=\E^{n-i,\gtw i}(X)=\E^{n-2i,\dtw i}(X);
\]
similar notation will be used for the associated bivariant theory.  When $\E=\un_S$, the various cohomology groups will be referred to as $\T$-cohomology and bivariant $\T$-theory respectively, and we write $H^{n,i}(X,\T)$ (resp. $H_{n,i}^{BM}(X/S,\T)$) for these groups.

Write $\uK(S)$ for the category of virtual vector bundles over $S$,
 associated with the category of vector bundles over $S$ with morphisms the isomorphisms
 of vector bundles.
 The Thom space construction may be viewed as a functor
\[
\Th_S:\uK(S) \rightarrow \T(S)
\]
that sends sums to tensor products. Following \cite{DJK}, we may twist cohomology and bivariant theories by pairs $(n,v) \in \ZZ\times \uK(X)$; we use the following notation for these twists:
\begin{align*}
\E^{n}(X,v)&=\Hom_{\T(X)}(\un_X,f^*\E \otimes Th_X(v)[n]) \\
\E_{n}(X/S,v)&=\Hom_{\T(X)}\big(\Th_X(v)[n],f^!\E\big).
\end{align*}
Given an integer $i \in \ZZ$, we denote by $\dtw i$ the unique free virtual vector bundle of rank $i$ -- the underlying scheme is implicit -- so that this notation is compatible with our conventions on twists.  If $K$ is a perfect complex over $X$, we denote by $\dtw K$ the associated virtual bundle (see \cite[2.1.5]{DJK}).

In order to be able to apply the construction of the $\delta$-homotopy t-structure of \cite{BD1}, we will require that $\T$ satisfies the following assumptions:
\begin{enumerate}
 \item[(T1)] $\T$ is generated by Tate twists of smooth schemes: more precisely, for any scheme $S$, $\T(S)$ is generated as a triangulated category by objects of the form $M_S(X)(i)$ with $X/S$ smooth and $i \in \ZZ$.\footnote{All objects in $\T(S)$ are obtained by taking extensions of arbitrary coproducts. Equivalently, an object $K$ of $\T(S)$ is zero if and only if:
\[
\forall X/S \text{ smooth,} (n,i) \in \ZZ^2,
 \Hom_{\T(S)}\big(M_S(X)(i)[n],K\big)=0.
\]}
\item[(T2)] $\T$ is continuous with respect to Tate twists:
 see \cite[4.3.2]{CD3}.\footnote{Recall that this expresses the compatibility of $\T$
 with projective limits.}
\item[(T3)] $\T$ is {\it homotopically compatible}: see \cite[3.2.12]{BD1}.\footnote{This
 condition is automatically verified (see \cite[3.2.13]{BD1})
 if $\T$ satisfies absolute purity and the following vanishing statement holds:
\[
\forall \text{fields } E, \forall n>m, H^{n,m}(\spec(E),\T)=0.
\]}
\item[(T4)] A suitable form of resolution of singularities holds: property
 (Resol) of \cite[2.4.1]{BD1}.\footnote{Under two geometric assumptions,
 this means that:
\begin{itemize}
\item In case (S1), we will require that, for $p$ the characteristic exponent of $k$, $\T$ is $\ZZ[1/p]$linear.
\item In case (S2), we require that $\T$ is $\QQ$-linear.
\end{itemize}}
\end{enumerate}

\noindent \textbf{Examples}.  The abstract setting being given,
 we now give our list of concrete frameworks, that will be used
 in all of our examples. We first fix some absolute base $S_0$
 and restrict our schemes to $S_0$-schemes essentially of finite type.
\begin{enumerate}
\item \textbf{Motivic case}. Let $R$ be a ring of coefficients,
 and assume one of the following situation:
\begin{itemize}
\item $S_0$ is the spectrum of a field $k$ whose characteristic exponent is invertible in $R$ and $\T = \DM(-,R)$ is Voevodsky's cdh-local category of triangulated mixed motives, introduced in \cite{CD5}.
\item $S_0$ is a Dedekind ring (or more generally an excellent scheme of dimension greater or equal than $3$),
 $R$ is a $\QQ$-algebra and $\T = \DM(-,R)$ is one the many models of the triangulated category of $R$-motives introduced in \cite{CD3}.
\end{itemize}
\item \textbf{Homotopical case}. The absolute base $S_0$ is the spectrum of a field with characteristic exponent $p$,
 and $\T=\SH[p^{-1}]$ is the $\ZZ[p^{-1}]$-linearization of the Morel-Voevodsky stable homotopy category of ${\mathbb P}^1$-spectra.  Particular motivic ring spectra of interest to us will include:
\begin{itemize}
\item  $S^0_S=\un_S$ the sphere spectrum.
\item $\HH R_S$ the motivic Eilenberg-MacLane ring spectrum with coefficients in $R$.\footnote{Recall this ring spectrum is obtained as the image of the constant motive under the canonical map: $K \circ \gamma_*:\DM(S,\ZZ[p^{-1}]) \rightarrow SH(S)[p^{-1}]$, by forgetting the transfer and then taking the Nisnevich Eilenberg-MacLane
 functor. See for example \cite{CD3}.}
\item $\HH \tilde R_S$ the Milnor--Witt motivic ring spectrum with coefficients in $R$.
 See \cite{DF2}.
\end{itemize}
\end{enumerate}
To fix ideas, the reader may take the dimension function $\delta_0$ on $S_0$ satisfying the conditions (S1) or (S2) fixed above, and for any $S_0$-scheme essentially finite type, take the induced dimension function (as described above).

\section{Homotopy t-structure and duality}
\label{sec:homtopytstructure}
  The goal of this section is to review a variant of the homotopy $t$-structure analyzed in \cite{BD1} that makes sense over rather general base schemes.  In more detail, Section~\ref{subsec:recollections} reviews the necessary theory from \cite{BD1} and fixes additional notations and conventions to be used in the sequel.  We highlight here Theorem~\ref{thm:deltahomotopytstructure} which summarizes the required existence result and Point~\ref{num:homologicalconventions} which describes {\em homological} conventions for $t$-structures that will be in force throughout the paper.  On the other hand, Section~\ref{subsec:purityandduality} is somewhat orthogonal to the above considerations and is concerned with recalling some facts about purity and duality that will be used in the remainder of the paper.

\subsection{Recollections on the homotopy {$t$}-structure}
\label{subsec:recollections}

\begin{num}\label{num:extension}
For any $\T$-spectrum $\E$ over a scheme $S$, bivariant $\E$-theory extends canonically from the category of separated $S$-schemes of finite type to that of separated $S$-schemes essentially of finite type.  Indeed, suppose $X$ is a separated $S$-scheme essentially of finite type. By assumption, there exists a pro-scheme $(X_\lambda)_\lambda$ with affine and \'etale transition morphisms, such that each $X_\lambda$ is separated and has finite type over $S$, and with limit $X$. We set:
\[
\E_{n,i}(X/S) := \ilim_{\lambda} \E_{n,i}(X_\lambda/S).
\]
This definition is independent of the pro-scheme presenting $X$ as a limit \cite[\S 8.2]{EGA4}.  The resulting definition presents a canonical extension of the original functor by the continuity assumption on $\T$ (i.e., property (T2) from our conventions).  Moreover, when $X$ has finite type over $S$ the new definition agrees with the old definition by using again the continuity assumption on $\T$.
\end{num}

\begin{prop}
\label{prop:homological_delta}
If $\E$ is a $\T$-spectrum over $S$, then the following conditions are equivalent.
\begin{enumerate}
\item[(i)] For any separated scheme $X/S$ of finite type, one has:
\begin{align*}
\E_{n,i}(X/S)=0
 \text{ when } & n-i<\delta_-(X), \\
\text{respectively } & n-i>\delta_+(X).
\end{align*}
\item[(ii)] For any point $x \in S(K)$ (see our Notations and conventions), one has:
\begin{align*}
\E_{n,i}(x)=0
 \text{ when } & n-i<\delta(x), \\
\text{respectively } & n-i>\delta(x).
\end{align*}
\end{enumerate}
\end{prop}

\begin{proof}
The equivalence of (i) and (ii) follows by combining Theorem 3.3.1 and Corollary 3.3.5 of \cite{BD1}.  Alternatively, it is a straightforward consequence of the existence and convergence of the $\delta$-niveau spectral sequence (\cite[Def. 3.1.5]{BD1} or Paragraph \ref{num:delta-niveau} in this paper).
\end{proof}

\begin{rem}
\begin{enumerate}
\item In \cite{BD1}, the extension $\E_{**}$ to separated schemes essentially of finite type was denoted by $\hat \E_{**}$. Since, according to paragraph \ref{num:extension}, this extension is unique and well-defined -- using in particular the assumption (T2) for coherence, we will not follow this notational convention here (and we caution the reader that the decoration $\hat{\E}$ is used with a different meaning in this paper. )
\item If we use the $\delta$-niveau spectral sequence, then the proof of the previous proposition does not use the assumptions (T1), (T3) and (T4).  Therefore, the preceding proposition is true without assuming these conditions.
\end{enumerate}
\end{rem}

We can now state the main theorem of \cite{BD1} (see \emph{loc. cit.} Th. 3.3.1 and Cor. 3.3.5).

\begin{thm}
\label{thm:deltahomotopytstructure}
Given any scheme $S$, there exists a $t$-structure on $\T(S)$ whose homologically non-negative (resp. non-positive) objects are the $\T$-spectra $\E$ over $S$ satisfying the equivalent conditions (i) and (ii) of the above proposition.  This $t$-structure is, moreover, non-degenerate and satisfies gluing in the sense of \cite[1.4.10]{BBD} (see also Remark \ref{rem:glueing}).
\end{thm}

\begin{df}
\label{df:htp_t}
Given any scheme $S$, the $t$-structure on $\T(S)$ of the above theorem will be called the \emph{$\delta$-homotopy t-structure}.  Objects of the heart of this $t$-structure, denoted by $\hrt{\T(S)}$, will be called \emph{$\delta$-homotopy modules}, or simply \emph{homotopy modules}
 when this does not lead to confusion.
\end{df}

\begin{rem}
\label{rem:delta_independence}
The $t$-category $\T(S)$ is ``independent" of the choice of $\delta$ in a sense we now explain.  Given another choice $\delta'$, we know that the function $\delta'-\delta$ is constant on connected components of $S$.  As $\T(S \sqcup S')=\T(S) \oplus \T(S')$, we may obviously assume $S$ is connected so that $\delta'=\delta+n$.  It follows from the above definition that
\[
\phi_{\delta,\delta'}:(\T(S),t_\delta) \rightarrow (\T(S),t_{\delta'}),
 \E\mapsto \E[n]
\]
is an equivalence of $t$-categories. It may be useful to remember the formula:
\begin{equation}\label{eq:delta_independence}
\tau_{\geq p}^{\delta+n}=\tau_{\geq p+n}^\delta
\end{equation}
(see the conventions of Par. \ref{num:homologicalconventions} for
 these truncation functors).
\end{rem}

\begin{ex}
\begin{enumerate}
\item \underline{Motivic case}:  Let $k$ be a perfect field and assume $\delta$ is the obvious dimension function on $k$.  In the motivic case, the $\delta$-homotopy $t$-structure on $\DM(k,R)$ coincides with the stable version of Voevodsky's homotopy $t$-structure introduced in \cite[Sec. 5.2]{Deg9}. See \cite[Ex. 2.3.5]{BD1} for details.  In particular, a $\delta$-homotopy module $\E$ over $k$ is just a $\ZZ$-graded homotopy invariant sheaf with transfers equipped with an isomorphism:
\[
(\E_{n+1})_{-1} \simeq \E_n.
\]
We refer the reader to \cite[1.17]{Deg9} for more details.
\item \underline{Homotopical case}:
In the homotopical case, given a field $k$ with the obvious dimension function, the $\delta$-homotopy $t$-structure on $\SH(k)$ coincides with Morel's homotopy $t$-structure (see \cite[Sec. 5.2]{Mor1}).  We refer the reader again to \cite[Ex. 2.3.5]{BD1} for more details.  In this case, a $\delta$-homotopy module $\E$ over $k$ is a $\ZZ$-graded
strictly $\AA^1$-invariant Nisnevich sheaf over smooth $k$-schemes with a given isomorphism:
\[
(\E_{n+1})_{-1} \simeq \E_n.
\]
In other word, this is a homotopy module in the sense of Morel.
Recall also that a spectrum $\E$ is called {\em orientable} if it admits the structure of a module over the ring spectrum $\mathbf{MGL}$.  Orientability for homotopy modules turns out to be equivalent to requiring that the Hopf map $\eta$ (an element in the graded endomorphisms of the motivic sphere spectrum) acts trivially.  In fact, orientability is also equivalent to requiring that $\E$ admits transfers, in which case these transfers are unique (see \cite[4.1.5, 4.1.7]{Deg10} for further details).
\item Over a general base $S$, in both the homotopical and motivic cases, the $\delta$-homotopy $t$-structure can be compared with the
 \emph{perverse homotopy $t$-structure} defined by Ayoub in \cite[\textsection 2.2.4]{Ayoub1}.  This comparison requires an appropriate choice of $\delta$, and we refer the reader to \cite[2.3.11]{BD1} for details.
\end{enumerate}
\end{ex}

\begin{num}
\label{num:homologicalconventions}
\textit{Homological conventions}.--
For the most part, we adopt homological conventions, as they are better suited to issues that arise involving singularities.  We will write $\E \geq_\delta n$ (resp. $\E\leq_\delta n$) to say that $\E$ is concentrated in homological degree above $n-1$ (resp. below $n+1$)
and denote by $\tau^\delta_{\geq n}$ (resp. $\tau^\delta_{\geq n}$) the corresponding homological truncation functor.  Using homological conventions, the truncation triangles read:
\[
\tau^\delta_{\geq 0}(\E) \rightarrow \E \rightarrow \tau^\delta_{<0}(\E) \xrightarrow{+1}
\]
and $\Hom(\E,\F)=0$ if $\E\geq_\delta 0$ and $\F<_\delta 0$.

We denote by $H_n^\delta$ the $n$-th homology functor with respect to the $t_\delta$-homotopy $t$-structure.  Finally, we summarize conventions with respect to suspensions:
\[
H_n^\delta(\E[i])=H_{n-i}^\delta(\E), \
 \tau_{\geq n}^\delta(\E[i])=\tau_{\geq n-i}^\delta(\E)[i], \
 \tau_{\leq n}^\delta(\E[i])=\tau_{\leq n-i}^\delta(\E)[i].
\]
\end{num}

\begin{rem}
We may pass from homological to cohomological indexing by writing indices as superscripts and reversing signs.  In formulas:
\[
H^n_\delta(\E)=H_{-n}^\delta(\E),
 \ \tau_\delta^{\leq n}=\tau^\delta_{\geq -n},
 \ \tau_\delta^{>n}=\tau^\delta_{<-n}.
\]
\end{rem}

\begin{num}\label{num:homological_delta_ref}
For the record, we now give a number of equivalent characterizations of positivity or negativity with respect to the $\delta$-homotopy $t$-structure -- this is simply a consequence of Proposition~\ref{prop:homological_delta}, using the above conventions and the construction in Paragraph~\ref{num:extension}.

Given a $\T$-spectrum $\E$ over $S$ and an integer $m \in \ZZ$, the following conditions are equivalent:
\begin{enumerate}
\item[(i)] $\E \geq_\delta m$ (resp. $\E \leq_\delta m$)
\item[(ii)] For any separated scheme $X/S$ having finite type, one has:
\begin{align*}
\E_{n,i}(X/S)=0
 \text{ when } & n-i<m+\delta_-(X), \\
\text{respectively } & n-i>m+\delta_+(X).
\end{align*}
\item[(ii')] For any separated scheme $X/S$ being essentially of finite type, one has:
\begin{align*}
\E_{n,i}(X/S)=0
 \text{ when } & n-i<m+\delta_-(X), \\
\text{respectively } & n-i>m+\delta_+(X).
\end{align*}
\item[(iii)] For any point $x \in S(K)$ (see again our Notations and conventions),
 one has:
\begin{align*}
\E_{n,i}(x)=0
 \text{ when } & n-i<m+\delta(x), \\
\text{respectively } & n-i>m+\delta(x).
\end{align*}
\end{enumerate}
\end{num}


\begin{num}\label{num:basic_t_exact}
\textit{$t$-exactness of the six operations}.--
Given a functor $F$ between triangulated categories equipped with $t$-structures, one says that $F$ is left (resp. right) $t$-exact if it respects
 homologically negative (resp. positive) objects; such a functor $F$ is $t$-exact if it is both left and right $t$-exact.  One says that  $F$ has homological amplitude $[a,b]$ if for any object $\E$:
\begin{itemize}
\item $\E \geq 0 \Rightarrow F(\E)\geq a$.
\item $\E \leq 0 \Rightarrow F(\E)\leq b$.
\end{itemize}
Let $f:X \rightarrow S$ be a morphism essentially of finite type, and $d$ the maximum dimension of its fibers. We consider the $\delta^f$-homotopy $t$-structure on $\T(X)$, where $\delta^f$ is the dimension function on $X$ induced by that of $S$ with respect
 to the morphism $f$ (see Notations and conventions page \pageref{sec:notations}).
 Then one has the following results:
\begin{itemize}
\item $f^*[d]$ is right $t_\delta$-exact.
\item If $f$ is smooth, $f^*[d]$ is $t_\delta$-exact.
\item The functor $f_*$ has $t_\delta$-amplitude $[0,d]$.
\end{itemize}
The first, second and third points are respectively proved in \cite{BD1},
 2.1.6(3), 2.1.12, 3.3.7. \\
If in addition $f$ is separated of finite type, we get:
\begin{itemize}
\item $f_!$ is right $t_\delta$-exact.
\item $f^!$ is $t_\delta$-exact.
\item If $\delta \geq 0$ then $\otimes$ is right $t_\delta$-exact.
\end{itemize}
These points are respectively proved in \cite{BD1},
 2.1.6(1), 3.3.7(4), 2.1.6(2).
\end{num}

\begin{rem}
\label{rem:glueing}
As mentioned in the Theorem~\ref{thm:deltahomotopytstructure}, the $\delta$-homotopy $t$-structure satisfies gluing.  We recall here precisely what this means.  Consider a closed immersion $i:Z \rightarrow S$ with complementary open immersion $j:U \rightarrow S$ and a $\T$-spectrum $\E$ over $S$.  The following conditions are equivalent:
\begin{itemize}
\item $\E$ is homologically non-$t_\delta$-negative (resp. non-$t_\delta$-positive).
\item $j^*\E$ and $i^*\E$ are homologically non-$t_\delta$-negative (resp. $j^*\E$ and $i^!\E$ are homologically non-$t_\delta$-positive).
\end{itemize}
The equivalence of these conditions can be deduced from the localization property of $\T$ and the $t$-exactness stated above (see \cite[Cor. 2.1.9]{BD1}).
\end{rem}

Let us now explicitly state the following consequence of the result on the tensor product.
\begin{prop}\label{prop:homotopy_pairing}
Assume the dimension function $\delta$ on $S$ is non-negative.
Let $\E$ be a $\T$-spectrum over $S$.
 Then for any pair $(p,q) \in \ZZ^2$, one has a canonical pairing:
$$
\bar \phi:\tau_{\geq p}^\delta(\E) \otimes \tau^\delta_{\geq q}(\F)
 \rightarrow \tau^\delta_{\geq p+q}(\E \otimes \F)
$$
which is bi-functorial in $\E$.
\end{prop}
\begin{proof}
According to the preceding paragraph, the assumption implies that
 the tensor product $\otimes$ in $\T(S)$ preserves non-negative objects.
Consider the canonical map
$$
\tau_{\geq p}^\delta(\E) \otimes \tau^\delta_{\geq q}(\F)
 \xrightarrow \phi \E \otimes \F.
$$
According to the preceding assertion, the left hand-side is in homological degrees
 $\geq p+q$. Consider the distinguished triangle:
$$
\tau_{\geq {p+q}}^\delta(\E \otimes \F)
 \xrightarrow{a} \E \otimes \F
 \xrightarrow{b} \tau_{<{p+q}}^\delta(\E \otimes \F) \xrightarrow{+1}
$$
We deduce that the composition $b \circ \phi$ is zero. So $\phi$ uniquely
 factors through $a$ giving us the desired map $\bar \phi$. The bifunctoriality of $\bar \phi$ follows from the uniqueness.
\end{proof}

\begin{ex}\label{ex:heart_monoidal}
Consider the assumptions of the preceding proposition.

The previous pairing is associative in an obvious sense
 - this follows from the uniqueness of the map $\bar \phi$.
 Hence the symmetric monoidal structure on $\T(X)$ induces
 a canonical symmetric monoidal structure on $\hrt{\T(X)}$, using the formula,
 for $\delta$-homotopy modules $\E$ and $\F$:
$$
\E \otimes^H \F:=\tau_{\leq 0}^\delta (\E\otimes \F).
$$
Note in particular that the canonical functor:
$$
\T(X)_{\geq_\delta 0} \rightarrow \hrt{\T(X)},
 \E \mapsto \tau_{\leq 0}^\delta(\E)=H_0^\delta(\E)
$$
is monoidal.
\end{ex}

\begin{rem}
Beware that the monoidal structure defined above a priori depends on
 $\delta \geq 0$. To be more precise,
 though the heart for two different choices of dimension functions $\delta \geq 0$, $\delta' \geq 0$
 are  equivalent as additive categories, according to Remark \ref{rem:delta_independence}, this equivalence is not compatible
 in general with the tensor structures induced by $\delta$
 and $\delta'$.

Indeed, if $k$ is a field and we are in the motivic
 (or the homotopical) case, if we take a dimension function
 $\delta$ on $k$ such that $\delta(k)=n>0$, then the
  induced tensor product on the heart is just the zero bi-functor !
The situation of a positive dimensional base is more complicated
 though, and it seems there are as many non-trivial tensor structures
 induced as in the above example as the dimension of the base.
\end{rem}

\subsection{Recollection on purity and duality}
\label{subsec:purityandduality}

\begin{num}\label{num:purity}
Consider a $\T$-spectrum $\E$ over a scheme $S$.

Let us recall a construction from \cite{DJK}. Let $f:Y \rightarrow X$
 be a quasi-projective lci morphism of $S$-schemes with cotangent complex $L_f$.
 Let $\E_X$ be the pullback of $\E$ along $X/S$.
 Then we associate to $f$ a purity transformation (see \cite[4.3.1]{DJK});
 evaluated at the object $\E_X$, it gives a canonical map:
$$
\mathfrak p_f:f^*(\E_X) \otimes \Th(L_f) \rightarrow f^!(\E_X).
$$
where $\Th(L_f)$ is the Thom space associated with the perfect complex $L_f$.
The following definition extends classical considerations; in the motivic case,
 see \cite[4.3.9]{DJK}.
\end{num}
\begin{df}\label{df:purity}
Consider the above notations.

We will say that $\E$ is {\it absolutely pure} if for any quasi-projective morphism
 $f:Y \rightarrow X$ between regular schemes, the map $\mathfrak p_f$ is an isomorphism.

If $S$ is regular, we will say that $\E$ is {\it $S$-pure} if
 for any quasi-projective morphism $f:X \rightarrow S$ with $X$ regular,
 $\mathfrak p_f$ is an isomorphism.
\end{df}

\begin{ex}\label{ex:purity}
\begin{enumerate}
\item \underline{Motivic case}: the constant motive $\un_{S_0}$ is absolutely pure.
 Equivalently, $\un_S$ is $S$-pure for any regular scheme $S$.
\item \underline{Homotopical case}:
 Then any spectrum $\E$ over the base field $k$ is absolutely pure
 (see \cite[4.3.10(ii)]{DJK}).
 Equivalently, any spectrum $\E$ over a regular base $S$ is $S$-pure.
\end{enumerate}
\end{ex}

An almost immediate corollary of the $S$-purity assumption
 is the following duality statement.
\begin{prop}\label{prop:duality}
Let $f:X \rightarrow S$ be a morphism and $\E$ be a $\T$-spectrum over
 $S$. We assume one of the following hypothesis is fulfilled:
\begin{itemize}
\item $f$ is essentially smooth.
\item $f$ is essentially quasi-projective, $X$ and $S$ are regular
 and $\E$ is $S$-pure.
\end{itemize}
Then for any pair $(n,v) \in \ZZ \times \uK(X)$,
 the map $\pur_f$ gives an isomorphism:
\begin{equation}\label{eq:duality}
\E^n(X,v) \xrightarrow \sim \E_{-n}(X/S,\dtw{L_f}-v)
\end{equation}
which is contravariantly natural in $X$ with respect to \'etale maps.
\end{prop}
\begin{proof}
The case where $f$ is of finite type is tautological,
 while the contravariance with respect to \'etale map follows from
 the compatibility of $\pur_f$ with \'etale pullbacks
 (apply \cite[3.3.2(iii)]{DJK} in the case where $p$ is \'etale).

The general case is obtained using the previous one,
 together with the naturality with respect to \'etale maps,
 and the extension of bivariant theory described in \ref{num:extension}.
\end{proof}

\begin{rem}
\begin{enumerate}
\item This isomorphism can be described,
 at least when $f$ is of finite type,
 as the cap-product by the fundamental class $\eta_f \in \E_0(X/S,\dtw{L_f})$
 of $f$: see \cite[4.3.9, 2.3.14]{DJK}.
 In fact, the description of the above isomorphism in the general case
 immediately follows when one considers the obvious extension of
 the notion of fundamental classes to essentially quasi-projective
 lci morphisms.
\item Recall that if either $\E$ is an object of $\DM(S,R)$ or if
 $E$ is an $\MGL$-module\footnote{in other words, an oriented spectrum;
 here module is to be understood in the sense of the monoidal category
 $\SH(S)$ \emph{i.e.} in the weak homotopical sense.}
 over $S$ then the Thom isomorphism implies we can canonically identify
 the twist by a virtual vector bundle with the Tate twist by its rank.
 In particular, the above isomorphism takes the following classical
 form:
\begin{equation}\label{eq:duality_or}
\E^{n,i}(X) \xrightarrow \sim E_{2d-n,d-i}(X/S)
\end{equation}
where $d$ is the relative dimension of $f$.
\end{enumerate}
\end{rem}

\section{Fiber homology and Gersten complexes}
\label{sec:fiberhomology}
In this section, we investigate the heart of the homotopy $t$-structure discussed in the preceding section in greater detail.  Section~\ref{subsec:fiberdeltahomology} is concerned with studying an approximation to the notation of homology associated with the truncation functors for the homotopy $t$-structure.  Definition~\ref{df:delta-hlg} introduces the notion of fiber $\delta$-homology, which is esentially the restriction of $\delta$-homology to points.  This section concludes with Theorem~\ref{thm:effectivedeltahomotopytstructure}, which introduces an ``effective" variant of the $t$-stucture of Section~\ref{subsec:recollections}, ameliorating some unboundedness issues that naturally arise because we work in a stable context.  Section~\ref{subsec:Gerstencomplexes} recalls conventions for exact couples, and introduces Gersten complexes (see Definition~\ref{df:Gersten}), built out of fiber $\delta$-homology, which essentially form the $E_1$-page of the niveau spectral sequence.  We also give a detailed description of the differentials of the Gersten complexes and important formal properties of these complexes are summarized in Proposition~\ref{prop:Gersten}.  Finally, Section~\ref{sub:products} is devoted to analyzing multiplicative structure in these complexes and the associated spectral sequences.

\subsection{Fiber $\delta$-homology}
\label{subsec:fiberdeltahomology}

\begin{num}
Recall from the introduction that we have also considered $\GG$-twists on bivariant theory. Indeed these twists are more natural with respect to the $\delta$-homotopy $t$-structure, because the functor $-(1)[1]$ is $t_\delta$-exact.  Using this grading, one can reformulate Proposition \ref{prop:homological_delta} for a given $\T$-spectrum $\E$ over $S$ as the equivalence of the following conditions:
\begin{itemize}
\item[(i)] $\E \geq 0$ (resp. $\E \leq 0$).
\item[(ii)] For any separated scheme $X/S$ essentially of finite type, one has:
\begin{align*}
\E_{n,\gtw *}(X/S)=0
 \text{ when } & n<\delta_-(X)\ \big(\text{resp. } n>\delta_+(X)\big).
\end{align*}
\item[(iii)] For any point $x \in S(K)$, one has:
\begin{align*}
\E_{n,\gtw *}(x)=0
 \text{ when } & n<\delta(x)\ \big(\text{resp. } n>\delta(x)\big).
\end{align*}
\end{itemize}
\end{num}

In view of this characterization of the $\delta$-homotopy $t$-structure, we have adopted the following definition in \cite{BD1}.

\begin{df}
\label{df:delta-hlg}
Let $\E$ be a $\T$-spectrum over $S$.  We define the \emph{fiber $\delta$-homology} of $\E$ in degree
 $n \in \ZZ$ as the functor
\[
\hat H_n^\delta(\E):\pts(S) \rightarrow \ab^\ZZ,
 x \mapsto \big( \E_{\delta(x)+n,\gtw{\delta(x)-r}}(x) \big)_{r \in \ZZ}
\]
where $\pts(S)$ is the discrete category of points of $S$.  Dually, we define the \emph{fiber $\delta$-cohomology} of $\E$ as $\hat H^n_\delta(\E)=\hat H_{-n}^\delta(\E)$ (the $\GG$-grading does not change).

Given an object $\F$ of the $\delta$-homotopy heart, we set $\hat \F^\delta_*:=\hat H_0^\delta(\F)$.
\end{df}

Therefore $\E \geq 0$ (resp. $\E \leq 0$) if and only if $\hat H_n^\delta(\E)=0$ for $n<0$ (resp. $n>0$).

\begin{rem}\label{rem:Rost_modules}
Fiber $\delta$-homology is a good approximation of $\delta$-homology.  In fact, given an object $\F$ of the heart of the $\delta$-homotopy $t$-structure on $\T(S)$, the functor $\hat\F_*^\delta$ is an approximation of a cycle module in the sense of Rost.

In the motivic case, this idea can be turned into an equivalence of categories between the $\delta$-homotopy heart of $\DM(S,R)$ and the category of Rost $R$-linear cycle modules over $S$ in the sense of \cite{Ros}.  The details of such a theorem have not yet been written up, but see \cite{Deg17}.  In the homotopical case, a generalization of Rost's theory is in the work (see \cite{Feld}).

In any case, we have already obtained in \cite[4.2.2]{BD1} that, in the general case of an abstract triangulated motivic category $\T$ satisfying our general assumptions, the functor:
\[
\hrt{\T(S)} \rightarrow \mathrm{PSh}\big(\pts(S),R-\text{mod}^\ZZ\big),
 \F \mapsto \hat \F^\delta_*
\]
is conservative, exact and commmutes with colimits.  Similarly, the family of functors $(\hat H_n^\delta)_{n \in \ZZ}$
 is conservative on the whole category $\T(S)$.
\end{rem}

\begin{ex}\label{ex:unit_stable}
Let $S$ be a regular connected scheme, and put $d=\delta(S)$.  We fix a point $x:\spec(K) \rightarrow S$.
\begin{enumerate}
\item Abstractly, using any $S$-pure spectrum $\E$,
 one obtains, because of the duality isomorphism \eqref{eq:duality}
 and relation (D2) of dimension functions,
 a canonical isomorphism:
\[
\hat H_n^\delta(\E)_r(x)
 \simeq \E^{-n-r}\big(\spec(K),\dtw{L_{x/S}}-\dtw{\delta(x)}+\dtw r\big).
\]
\item Assume we are in the \underline{motivic case}. One then obtains (using the previous computation) a canonical isomorphism:
\[
\hat H^n_\delta(\un_S)_r(x)=H^{r-n-2d,r-d}_M(\spec(K),R).
\]
This follows because under our assumptions, motivic cohomology satisfies absolute
 purity, and is oriented.

Observe, for example, that:
\[
\hat H_d^\delta(\un_S)_*=\hat H^{-d}_\delta(\un_S)_*=K_*^M|_S,
\]
the restriction of the Milnor K-theory functor to the discrete category $\pts(S)$.

Additionally, the only vanishing that we have is: $\hat H_n^\delta(\un_S)=0$ if $n>d$. In other words, $\un_S$ is concentrated in $t_\delta$-homological degrees $]-\infty,d]$.  On the other hand, we know for a fact, at least when $S$ is a complex scheme, that for all $n \leq d$, $\hat H_n^\delta(S) \neq 0$ (this is due to the existence and non-triviality of polylogarithm elements; see \cite[Ex. 3.3.2]{BD1}).  So $\un_S$ is $t_\delta$-unbounded below.

\item Assume we are in the \underline{homotopical case}.  Then one obtains, using the computation of point (1),
 a non-canonical isomorphism:
\[
\hat H_n^\delta(\un_S)_r(x)
 \simeq \pi_{n+d}^{\AA^1}(S^0_K)_{r-d}[1/p].
\]
This is because absolute purity is automatic in the homotopical case:
 see point (2) of Example \ref{ex:purity}.

In particular, Morel's computation of the zeroth stable $\AA^1$-homotopy sheaf of the sphere gives an isomorphism, again non-canonical:
\[
\hat H_d^\delta(\un_S)_r(x) \simeq K_r^{MW}(K)
\]
where the right hand side is the $r$-th Milnor--Witt cohomology group of the field $K$.

Moreover, the zero sphere spectrum $\un_S$ is concentrated in homological degrees $]-\infty,d]$ for the $\delta$-homotopy $t$-structure.  Not much else can be said in this case since, when $S$ is a complex
 scheme, the rational homology of the zero sphere spectrum agrees with the rational homology of the motivic spectrum,
 which is non-trivial.  Thus, $\un_S$ is unbounded below for the $\delta$-homotopy $t$-structure.

\item Again in the homotopical case, we can consider the Milnor--Witt motivic ring spectrum
 $\HH \tilde R_S$.  Known computations of Milnor--Witt cohomology imply, as in the case of motivic cohomology, that $\HH\tilde R_S$ is concentrated in homological degrees $]-\infty,d]$.  Moreover, the unit of this ring spectrum induces a canonical isomorphism:
\[
\hat H_d^\delta(\un_S)_* \xrightarrow \sim \hat H_d^\delta(\HH \tilde R_S)_*.
\]
\end{enumerate}
\end{ex}

\begin{rem}
In the non-oriented context, it appears more natural to introduce another twist for a virtual vector bundle $v$ over a scheme $X$, namely:
\[
\un_X\gtw v:=\Th_X(v)[-\rk(v)].
\]
This notation is consistent with our notation $\un_X\gtw r$ for the $\GG$-twist (where $r$ corresponds to the trivial bundle of rank $r$).

Assuming $\delta(S)=0$ to simplify, the isomorphism of point (1) can be rewritten as:
\[
\hat H_n^\delta(\E)_r(x)
 \simeq \E^{-n}\big(\spec(K),\gtw{L_{x/S}}+\gtw r\big).
\]
\end{rem}

The unboundedness of the constant object, observed in Points (2), (3), (4) of the previous example, appears to be a consequence of working in the stable context -- recall that stability with respect to Tate twists is among the axioms of triangulated motivic categories used in \cite{CD3}.  This potentially unpleasant feature can be corrected as follows.

\begin{df}\label{df:delta-eff}
(see \cite[Def. 2.2.1]{BD1})
Let $S$ be an arbitrary scheme. We define the triangulated category of $\delta$-effective $\T$-spectra over $S$, denoted by $\T^{\delta-\eff}(S)$, as the full localizing subcategory of $\T(S)$ generated by objects of the form:
\[
f_!(\un_X)(n)
\]
where $f:X \rightarrow S$ is separated of finite type and $\delta(X)\geq n$.

Given a $\T$-spectrum $\E$ over $S$, we define its \emph{effective fiber $\delta$-homology $H_n^{\delta\eff}(\E)$} as the negatively graded functor obtained from $\hat H_n^{\delta}(\E)$ by restricting the $\GG$-grading to negative integers.
\end{df}

It follows from \cite[3.1.1]{BD1} that one can define an effective version of the $\delta$-homotopy $t$-structure.

\begin{thm}
\label{thm:effectivedeltahomotopytstructure}
Consider the notations of the previous definition.  For any scheme $S$, there exists a $t$-structure on $\T^{\delta\eff}(S)$ whose homologically non-negative (resp. non-positive) objects are those $\T$-spectra $\E$ over $S$ such that $\hat H_n^{\delta\eff}(\E)=0$ for $n<0$ (resp. $n>0$).  This $t$-structure is non-degenerate and satisfies glueing (recall: \cite[1.4.10]{BBD}, Remark \ref{rem:glueing}).
\end{thm}

\begin{ex}\label{ex:unit_eff}
Let $S$ be a regular connected scheme, and put $d=\delta(S)$.
\begin{enumerate}
\item Assume we are in the \underline{motivic case}.
 Then the computation done in Example \ref{ex:unit_stable}(1)
 shows that $\hat H_n^\delta(\un_S)_{\leq 0}=0$ if $n \neq d$.
 In other words, $\un_S[-d]$ is in the heart of the effective $\delta$-homotopy
 $t$-structure. This is exactly what happens for the perverse
 $t$-structure.
\item Assume we are in the \underline{homotopical case}.
 The preceding example shows that the Eilenberg-MacLane motivic ring
 spectrum $\HH R_S[-d]$ is in the heart of the effective $\delta$-homotopy
 $t$-structure.

The computations of higher stable homotopy groups of spheres do
 not allow us at the moment to conclude anything about the
 sphere spectrum. However, the computations of Milnor-Witt motivic
 cohomology imply, as in the motivic case, that $\HH_{MW}R_S[-d]$
 is in the heart of the effective $\delta$-homotopy
 $t$-structure.
\end{enumerate}
\end{ex}

\begin{num}\label{num:delta-effective}
The $\delta$-effective categories $\T^{\delta\eff}(S)$, with the $t$-structure just defined, satisfy good properties. We refer the reader to \cite{BD1} for the following facts.  First, by construction, one has a pair of adjoint functors:
\[
s:\T^{\delta\eff}(S) \leftrightarrows \T(S):w
\]
such that $s$ is fully faithful and $w$ is $t_\delta$-exact.
\begin{itemize}
\item If $\delta \geq 0$, the subcategory $\T^{\delta\eff}(S)$ of $\T(S)$ is stable
 under tensor products. So $\T^{\delta\eff}(S)$ becomes a closed symmetric monoidal category
 with internal hom given by:
\[
\uHom_{\T^{\delta\eff}}(S)=w \circ \uHom_{\T(S)}(M,N).
\]
Moreover, the tensor product on $\T^{\delta\eff}(S)$ is right $t_\delta$-exact.
\item For $f:X \rightarrow S$ essentially of finite type with $\dim(f)\leq d$, we get an adjunction of $t$-categories:
\[
f^*(d)[2d]:\T^{\delta\eff}(S) \leftrightarrows \T^{\delta\eff}(X):w \circ \big(f_*(-d)[-2d] \big).
\]
When $f$ is smooth of pure dimension $d$, $f^*(d)[2d]$ is $t_\delta$-exact.
\item For $f:X \rightarrow S$ separated and essentially of finite type,
 we get an adjunction of $t$-categories:
\[
f_!:\T^{\delta\eff}(X) \leftrightarrows \T^{\delta\eff}(S):w \circ f^!.
\]
Moreover, $w \circ f^!$ is $t_\delta$-exact.
\end{itemize}
In the last two points, the dimension function on $X$ is that induced by the one fixed on $S$ (see our Notations and conventions).
\end{num}


\subsection{Gersten complexes}
\label{subsec:Gerstencomplexes}
We now build a spectral sequence using the theory of exact couples.  Since conventions for exact couples vary in the literature, we now fix our conventions, which agree with those used in \cite[3.1.2]{BD1}.

\begin{df}
\label{df:exact_couple}
Let $\A$ be an abelian category.  A {\it homological exact couple of degree $d>0$ in $\A$} is the data of a pair of bigraded objects $(D,E)$ of $\A$ together with a triangle of homogeneous maps $(a,b,c)$
\[
\xymatrix@=30pt{
D\ar_a^{(-1,+1)}[rr] && D\ar_/4pt/b^{(0,0)}[ld] \\
 & E\ar_-/-4pt/c^{(-d,d-1)}[lu] &
}
\]
such that consecutive maps fit together to form an exact sequence.
\end{df}

As usual, one derives a (homologically indexed) spectral sequence from such an exact couple \cite{McC}: at page $r$, the $r$-th term is equal to the differential bigraded abelian group $(E,b \circ c)$.


\begin{num}\label{num:delta-niveau}
We first recall $\delta$-niveau spectral sequences,  in the abstract setting and using dimension function following \cite[3.1.5]{BD1}.\footnote{The main idea is of course classical, introduced by Grothendieck, and thoroughly developed in \cite{BO}. The main novelty of \cite{BD1} is to work over an arbitrary base scheme $S$ (rather than a field)
 and appeal to abstract dimension functions.  This use of dimension functions can also be found in the homological indexing for Chow groups in \cite{stack}. See \cite[3.1.7]{BD1} for more on this point.}

Let $X$ be a separated $S$-scheme essentially of finite type.  Recall from \cite[3.1.1]{BD1} that a $\delta$-flag of $X$ is an increasing sequence of reduced closed subschemes $Z_*=(Z_p)_{p \in \ZZ}$ of $X$ such that $\delta(Z_p)\leq p$ for all $p$. The set $\flag(X)$ of $\delta$-flags, ordered by term-wise inclusion, is cofiltered. Given such a $\delta$-flag, the classical properties of the bivariant theory $\E_{*}(-/-)$ 
 (see our Notations and conventions) imply that we have a long exact sequence:
\[
\E_{p+q}(Z_{p-1}/S) \xrightarrow a \E_{p+q}(Z_{p}/S)
 \xrightarrow b \E_{p+q}(Z_{p}-Z_{p-1}/S)
 \xrightarrow c \E_{p+q-1}(Z_{p-1}/S)
\]
where $a$ (resp. $b$) is pushfoward (resp. pullback) along the obvious closed (resp. open) immersion, and $c$ the boundary map. These long exact sequences are covariantly functorial with respect to inclusion of flags. Thus we can combine them into the following homological exact couple of degree $1$:
\begin{align*}
{}^\delta D_{p,q}&=\ilim_{Z_* \in \flag(X)} \big(\E_{p+q}(Z_{p}/S)\big), \\
{}^\delta E^1_{p,q}&=\ilim_{Z_* \in \flag(X)} \big(\E_{p+q}(Z_{p}-Z_{p-1}/S)\big).
\end{align*}
Finally, observe that one can express the $E_1$-term as follows:
\[
{}^\delta E^1_{p,q}=\oplus_{x \in X_{(p)}} \E_{p+q}(x/S),
\]
where
\[
X_{(p)}=\{ x \in X \mid \delta(x)=p \}.
\]
\end{num}

\begin{df}\label{df:delta_niveau}
Under the assumptions and notations above, the spectral sequence:
\[
{}^\delta E^1_{p,q}=\bigoplus_{x \in X_{(p)}} \E_{p+q}(x/S) \Rightarrow \E_{p+q}(X/S)
\]
will be called the $\delta$-niveau spectral sequence of $X/S$ with coefficients in $\E$.

The spectral sequence converges to the following filtration, called the $\delta$-niveau filtration:
\[
N_p\E_*(X/S)=\bigcup_{i:Z \rightarrow X, \delta(Z)=p}
 \mathrm{Im}\big(i_*:\E_*(Z/S) \rightarrow \E_*(X/S)\big),
\]
where $i$ runs over the closed immersions whose source has $\delta$-dimension $p$.
\end{df}

It is standard to consider this spectral sequence not just for $\E$ itself, but also for all Tate twists.  In fact, if we apply this definition to the graded spectrum $\E(n)$ for an integer $n \in \ZZ$, we get the following form:
\[
{}^\delta E^1_{p,q}=\bigoplus_{x \in X_{(p)}} \E_{p+q,n}(x/S) \Rightarrow \E_{p+q,n}(X/S).
\]

In view of Rost's theory of cycle modules, it will be useful to introduce the following definition.

\begin{df}
\label{df:Gersten}
Under the assumptions of the preceding definition, we define the \emph{Gersten ($\delta$-homologi\-cal) complex}
 of $X/S$ with coefficients in $\E$, denoted by $C_*^\delta(X,\E)$, as the complex of abelian groups located at the line $q=0$ of the $E_1$-term of the $\delta$-niveau spectral sequence.  The $p$-th homology of this complex will be called the \emph{Gersten $\delta$-homology of $\E$}
 and be denoted by:
\[
A_p^\delta(X,\E)=H_p\big(C_*^\delta(X,\E)\big).
\]
\end{df}

Note that the Gersten complex is concentrated in degrees $[\delta_-(X),\delta_+(X)]$.  Moreover, using Definition \ref{df:delta-hlg} for fiber $\delta$-homology, the $p$-th term of this complex takes the following form:
\[
C_p^\delta(X,\E) = \bigoplus_{x \in X_{(p)}} H_0^{\delta}\E_p(x/S).
\]

\begin{rem}
The complex defined above is closely linked with Rost's theory of cycle modules.  We have used a simplification here. Indeed recall that cycle modules, as well as their associated complexes, are $\ZZ$-graded. We can recover this $\ZZ$-grading by applying the above definition to the $\ZZ$-graded spectrum:
\[
\E\{*\}=\big( \E(n)[n] \big)_{n \in \ZZ}.
\]
It is also standard in stable $\aone$-homotopy theory to consider this grading, called the {\it $\GG$-grading}.

The complex $C_*^\delta(X,\E\{*\})$ is $\ZZ$-graded, and has the following form:
\[
C_*^\delta(X,\E\{*\})=\bigoplus_{x \in X_{(p)}} H_0^{\delta}\E_{p-*}(x/S).
\]
The main reason to use this grading is that the differentials of the Gersten complex are then homogeneous of degree $-1$.
\end{rem}

\begin{num}
\label{num:differentials}
Consider the notations of the above definition.  The differentials of the Gersten complex associated with $\E$ and $X/S$
 take the following form:
\[
C_p^\delta(X,\E)=\bigoplus_{x \in X_{(p)}} \E_p(x/S)
 \xrightarrow{d_p}
 \bigoplus_{s \in X_{(p-1)}} \E_p(s/S)=C_{p-1}^\delta(X,\E)
\]
The differential $d_p$ is obtained as the inductive limit of differentials $d_p^Z$ associated with a $\delta$-flag $Z_*$
 and is described as follows:
\[
\xymatrix@C=16pt@R=20pt{
\E_p(Z_p/S)\ar[r]
 & E_p(Z_p-Z_{p-1}/S)\ar^c[r]\ar_/-30pt/{d_p}[rrd]
 & E_{p-1}(Z_{p-1}/S)\ar@{=}[d] & \\
& E_{p-1}(Z_{p-2}/S)\ar[r]
 & E_{p-1}(Z_{p-1}/S)\ar_-b[r]
 & E_{p-1}(Z_{p-1}-Z_{p-3}/S)
}
\]
where the rows arise as pieces of localization long exact sequences.

Using the functoriality of the localization long exact sequences with respect to pullbacks along open immersions, it is therefore possible to explicitly describe the differentials of this complex as follows.  Consider a pair $(x,s) \in X_{(p)} \times X_{(p-1)}$ and denote by $(d_p)_y^x:\E_p(x/S)_{r-p} \rightarrow \E_{p-1}(y/S)_{r-p}$ the corresponding component of the above differential.  Let us write $Z(x)$ the reduced closure of $x$ in $X$.  From the construction of the Gersten complex given above, and the functoriality of localization long exact sequence in bivariant theory with respect to pullbacks along open immersions and pushforward along closed immersions, one then deduces that:
\begin{itemize}
\item if $s \in Z(x)$: we let $Z(x)_{(s)}$ be the localization of $Z(x)$
 at $s$; this is a $1$-dimensional scheme so that $Z(x)_{(s)}=\{x,s\}$.
 Then $(d_p)_s^x$ is the boundary map of the localization long exact sequence
 associated with the closed immersion $i:\{s\} \rightarrow Z(x)_{(s)}$.
 Explicitly, it is the middle map in the following exact sequence:
\[
\E_p\big(Z(x)_{(s)}/S\big) \xrightarrow{j^*} \E_p(x/S)
 \xrightarrow{(d_p)_s^x} \E_{p-1}(s/S) \xrightarrow{i_*} \E_{p-1}\big(Z(x)_{(s)}/S\big)
\]
\item Otherwise, $(d_p)_s^x=0$.
\end{itemize}
\end{num}

The next result summarizes the formal properties of Gersten complexes.

\begin{prop}
\label{prop:Gersten}
Consider the notations of the above definition.
\begin{enumerate}
\item The complex $C^\delta_*(X,\E)$ is covariantly functorial in $\E$.
 Given any integer $p$, the induced maps
\[
\ C^\delta_*(X,\E) \rightarrow C^\delta_*(X,\tau_{\leq p}\E) , \
C^\delta_*(X,\tau_{\geq -p}\E) \rightarrow C^\delta_*(X,\E)
\]
are isomorphisms provided $p \geq 0$. In particular, one has a canonical isomorphism:
\[
C^\delta_*(X,\E) \simeq C^\delta_*(X,H_0^\delta\E).
\]
\item The complex $C^\delta_*(X,\E)$ is functorial in $X/S$,
 covariantly with respect to proper maps and contravariantly with respect
 to \'etale maps.
\item If the $\T$-spectrum $\E$ is homologically non-$t_\delta$-negative
 (resp. non-$t_\delta$-positive),
 there exists a canonical epimorphism
 (resp. monomorphism):
\begin{align*}
&\E_{p}(X/S) \rightarrow A_p^\delta(X,\E) \\
&\text{resp. } A_p^\delta(X,\E) \rightarrow \E_{p}(X/S).
\end{align*}
 These two maps are functorial in the $S$-scheme $X$,
 covariantly with respect to proper maps and contravariantly with respect
 to \'etale maps.

If $\E$ is in the $\delta$-homotopy heart, or more generally is concentrated in one degree for the $\delta$-homotopy $t$-structure, these two maps are inverse isomorphisms giving a functorial identification:
\[
\E_{p}(X/S) \simeq A_p^\delta(X,\E).
\]
\end{enumerate}
\end{prop}
\begin{proof}
Point (1) follows from the $\delta$-niveau spectral sequence and its obvious functoriality in $\E$ (Paragraph \ref{num:delta-niveau}).

Point (2) follows from the classical functoriality of the ($\delta$-)niveau spectral sequences. We refer the reader to
 \cite{Jin1}, proof of Proposition 3.11 for proper functoriality, and proof of Proposition 3.12 for \'etale contravariance. Note an important technical point here: we deal with the case of a general triangulated category $\T$, in contrast with the special case $\T=\DM$ which is oriented. In the case of proper functoriality, this does not come into play. The case of \'etale contravariance works as well as the tangent bundle of an \'etale map is trivial.

Point (3) is a consequence of the convergence of the $\delta$-niveau spectral sequence and of the computation of its $E_1$-term.
\end{proof}

\begin{rem}
\begin{enumerate}
\item According to the isomorphism of Point (3) above, one obtains using the main result of
 \cite{DJK} that, given a $\delta$-homotopy module $\E$, the Gersten homology $A_*^\delta(X,\E)$
 is contravariant in $X$ with respect to any smoothable lci morphism $f:Y \rightarrow X$.
 More precisely, one gets a canonical morphism:
\[
A_*^\delta(X,\E) \simeq \E_{p}(X/S)
 \rightarrow \E_p(Y/S,\dtw{L_{Y/X}}) \simeq A_*^\delta(Y,\E\dtw{L_{Y/X}})
\]
where the last identification uses the fact $\E\dtw{L_{Y/X}}$
 is concentrated in one degree over any connected component of $Y$
 --- which is the virtual rank of $L_{Y/X,\eta}$ where $\eta$ is any generic point of the
 chosen connected component.
\item As another illustration of the connexion of our Gersten complexes
 with Rost's theory of cycle complexes, when $S$ is the spectrum of a perfect field $k$, $\delta$ is the obvious dimension function,
 and we are in the motivic case $\T=\DM(k,R)$.
 Then any object $\E$ in the heart of $\DM(k,R)$ is a homotopy module with transfers in the sense of \cite{Deg9} which canonically corresponds to a cycle modules $\hat \E_*=\hat \E_*^\delta$
 --- this notation corresponds to the one of Definition \ref{df:delta-hlg}.
 Moreover, according to \cite[2.7(ii)]{Deg11}, there exists a canonical isomorphism of complexes
\[
C_*^\delta(X,\E) \simeq C_*(X,\hat \E_*)
\]
where the right hand-side is Rost cycle complex associated with the cycle module $\hat \E_*$.
\end{enumerate}
\end{rem}

\subsection{Products}
\label{sec:products}

\begin{num}\label{num:products_Gesten_hlg}
Consider a $\T$-spectrum $\E$ over $S$.

In order to describe products, we will use cohomological notations. We put:
\[
C^p_\delta(S,\E)=C_{-p}^\delta(S,\E),
 \text{ resp. }
 A^p_\delta(S,\E):=A_{-p}^\delta(S,\E)
\]
and call it the \emph{Gersten $\delta$-cohomological complex} (resp. \emph{$\delta$-cohomology}) of $S$ with coefficients in $\E$.  Note in particular that, when $\E$ is in the heart, we get from Proposition~\ref{prop:Gersten} a canonical isomorphism:
\[
A^p_\delta(S,\E) \simeq H^p(S,\E).
\]
Consider now a morphism of $\T$-spectra over $S$:
\[
\mu:\E \otimes \F \rightarrow \G.
\]
We deduce as usual a morphism at the level of cohomologies:
\[
\E^p(S) \otimes_\ZZ\F^q(S) \rightarrow \G^{p+q}(S)
\]
by sending a pair of maps $(a:\un_S \rightarrow \E[p],b:\un_S \rightarrow \F[q])$ to the following map:
\[
\un_S \xrightarrow{a \otimes b} \E \otimes \F[p+q] \xrightarrow \mu \G[p+q].
\]
Using the preceding isomorphism, we deduce a canonical pairing:
\[
A^p_\delta(S,\E) \otimes_\ZZ A^q_\delta(S,\F) \rightarrow A^{p+q}_\delta(S,\G).
\]
We deduce the following results.
\end{num}

\begin{prop}
\label{prop:products_Gesten_hlg}
If $\E$ is a ring $\T$-spectrum over $S$, then $A^*_\delta(S,\E)$ has a ring structure.  If $\F$ is a $\T$-spectrum with equipped with an $\E$-module-structure, then $A^*(S,\F)$ is equipped with the structure of a module over $A^*_\delta(S,\E)$.
\end{prop}

\begin{ex}\label{ex:products_Gersten&d-homology}
If $\E$ be a ring $\T$-spectrum over $S$, then according to Example \ref{ex:heart_monoidal}, the graded $\delta$-homotopy module $H_*^\delta(\E)$ has the structure of a ring spectrum. According to the preceding
 proposition, $A^*_\delta(S,H_*^\delta(\E))$ has the structure of a bigraded ring.
\end{ex}

\begin{rem}
\begin{enumerate}
\item Products are thus easy to obtain on our Gersten $\delta$-homology.  It is possible to follow Rost's approach in \cite{Ros} and to get a definition of the product on the level of complexes.  Using classical techniques due to Levine, it is even possible to find a dg-algebra underlying our Gersten $\delta$-homology (see \cite{lev1, Lev2, BaYa}).  An advantage of our approach is that it circumvents these technicalities.
\item One can also extend the preceding considerations to the case of smooth $S$-schemes. Indeed, in that case, one has:
\[
A^p_\delta(X,\E)=A_{-p}^\delta(X,\E) \simeq \E_{-p}(X/S)=H^p(X,\E\dtw{L_{X/S}})
\]
according to Proposition \ref{prop:duality}. The preceding proposition obviously extends to that case. Note that the product obtained for smooth $S$-schemes is now compatible with \'etale pullbacks as defined in Proposition \ref{prop:Gersten}, and even with respect to smooth pullbacks.
\end{enumerate}
\end{rem}

\section{The homotopy Leray spectral sequences}
\label{sec:Leray_ssq}
This section contains the main theoretical results of the paper and constructs the spectral sequences we mentioned in the introduction.  We introduce two versions of this spectral sequence.  Section~\ref{subsec:homologicalversion} studies a homological version of the Leray spectral sequence; the main result is Theorem~\ref{thm:Leray_ssp}, which gives a description of the $E_2$-page of the relevant spectral sequence.  Proposition~\ref{prop:formalpropertieshomologicallerayss} summarizes formal properties of the resulting spectral sequences.  On the other hand, Section~\ref{sec:coh_version} contains a cohomological version of the spectral sequence (summarized in Theorem~\ref{thm:deltahomotopylerayss}) and discusses compatibility with product structures.  Finally, Section~\ref{subsec:simplicity} contains a discussion of an analog of locally constant sheaves in our situation.

\subsection{The homological version}
\label{subsec:homologicalversion}
\begin{num}
\label{num:Leray_ssp}
Let us fix a $\T$-spectrum $\E$ over $S$ and consider the following geometric situation:
\[
\xymatrix@=15pt{
X\ar^f[rr]\ar_\pi[rd] && B \\
& S &
}
\]
where $f$ is any morphism of schemes, which plays the role of the fibration.
 We assume $\pi$ is a separated morphism essentially of finite type and put $\E^!_X=\pi^!\E$.
 Given a $\T$-spectrum $\E$ over $B$, we put according to what was
 announced in the paragraph on Notations and conventions:
\[
H_p(B,\E)=\Hom_{\T(B)}(\un_B[p],\E).
\]

Then one can look at the tower of homological truncations of $f_*(\E_X^!)$
 for the $\delta$-homotopy $t$-structure, which is the analogue of the Postnikov tower
 in our situation:
\[
\hdots \rightarrow \tau^\delta_{\geq q+1}(f_*\E_X^!) \rightarrow \tau^\delta_{\geq q}(f_*\E_X^!)
 \rightarrow \tau^\delta_{\geq q-1}(f_*\E_X^!) \rightarrow \hdots
\]
It is standard to deduce from this filtration a spectral sequence.  Let us be more precise. We first consider the canonical distinguished triangle:
\[
\tau^\delta_{\geq q+1}(f_*\E_X^!) \rightarrow \tau^\delta_{\geq q}(f_*\E_X^!) \rightarrow
 \tau^\delta_{=q}(f_*\E_X^!) \rightarrow \tau^\delta_{\geq q+1}(f_*\E_X^!)[1]
\]
where:
\[
\tau^\delta_{=q}(f_*\E_X^!)=\tau^\delta_{\leq q}\tau^\delta_{\geq q}(f_*\E_X^!)=H_q^\delta(f_*\E_X^!)[q].
\]
Applying the functor $H_{p+q}(B,-)$ to the preceding distinguished triangle, we get a long exact sequence:
\[
H_{p+q}\left(B,\tau^\delta_{\geq q+1}(f_*\E_X^!)\right)
 \xrightarrow a H_{p+q}\left(B,\tau^\delta_{\geq q}(f_*\E_X^!)\right)
 \xrightarrow b H_{p}\left(B,H_q^\delta(f_*\E_X^!)\right)
 \xrightarrow c H_{p+q-1}\left(B,\tau^\delta_{\geq q+1}(f_*\E_X^!)\right).
\]
Using the conventions of definition \ref{df:exact_couple}, we then get a homological exact couple of degree $2$ such that:
\begin{align*}
D_{p,q}&=H_{p+q}\left(B,\tau^\delta_{\geq q}(f_*\E_X^!))\right), \\
E_{p,q}&=H_{p}\left(B,H^\delta_{q}(f_*\E_X^!)\right).
\end{align*}
We deduce from that exact couple our main construction.
\end{num}

\begin{thm}
\label{thm:Leray_ssp}
Fix assumptions and notations as in \ref{num:Leray_ssp}.  The exact couple defined above gives a convergent spectral sequence of the form:
\[
E^2_{p,q}(f,\E)=H_{p}\left(B,H^\delta_{q}(f_*\E_X^!)\right)
 \simeq A_p^\delta\big(B,H^\delta_{q}(f_*\E_X^!))\big)
 \Rightarrow \E_{p+q}(X/S).
\]
The $\E_2$-term is concentrated in the range $p\in [\delta_-(X),\delta_+(X)]$ and is the homology in degree $p$ of the $\delta$-homological Gersten complex $C_*^\delta\big(B,H^\delta_{q}(f_*\E))\big)$ (see Definition \ref{df:Gersten}) which takes the form:
\[
\hdots \rightarrow \bigoplus_{x \in B_{(p)}} \E_{p+q}(X_x/S)
 \xrightarrow{\ d_p^G\ } \bigoplus_{s \in B_{(p-1)}} \E_{p+q-1}(X_s/S)
\rightarrow \hdots
\]
where $X_x$ is the fiber of $f$ above the point $x \in B$.
\end{thm}

Note, in particular, that the differentials $d^r_{p,q}$ are trivial for $r>\big(\delta_+(X)-\delta_-(X)$\big).

\begin{proof}
The spectral sequence follows from the theory of (derived) exact couples.  One can compute the $E_2$-term as in the above statement by using Proposition \ref{prop:Gersten}:
\[
H_{p}\left(B,H^\delta_{q}(f_*\E_X^!)\right) \simeq A_{p}^\delta\left(B,H^\delta_{q}(f_*\E_X^!)\right).
\]
Going back to the definition, this is the $p$-th homology of the following Gersten $\delta$-homological complex:
\[
\hdots \rightarrow \bigoplus_{x \in B_{(p)}} \E_{p+q}(x/B,f_*\E_X^!)
 \xrightarrow{\ d_p^G\ } \bigoplus_{s \in B_{(p-1)}} \E_{p+q-1}(s/B,f_*\E_X^!)
\rightarrow \hdots
\]
Given any point $x \in B^{(p)}$ one considers the cartesian square:
\[
\xymatrix@=10pt{
X_x\ar^{i_x}[r]\ar_{f_x}[d] & X\ar^f[d] \\
\spec(\kappa_x)\ar^-x[r] & B
}
\]
The computation:
\[
\E_{*}(x/B,f_*\E_X^!)
=[x_!\un_B[*],f_*\E_X^!]=[\un_x[*],x^!f_*\E_X^!]
=[\un_x[*],f_{x*}i_x^!\E_X^!]=[\un_x[*],i_x^!\pi^!\E]=\E_{*}(X_x/S)
\]
yields the $E_2$-term in the form of the statement and implies that the spectral sequence converges.
\end{proof}

\begin{df}
\label{df:Leray_ssp}
The above spectral sequence will be called the \emph{(homological) homotopy Leray spectral sequence} associated with $f$
 and with coefficients in $\E$.
\end{df}

The spectral sequence is obviously functorial in the $\T$-spectrum $\E$ over $S$.  If we want to analyze the behavior with respect to Tate twists, we can look at the spectral sequence with coefficients in $\E(n)$, which takes the following form:
\[
E^2_{p,q}=A_p^\delta\big(B,H^\delta_{q}(f_*\E_X^!)(n))\big)
 \Rightarrow \E_{p+q,n}(X/S).
\]
Similarly, twists by more general Thom spaces man be considered.  The following result summarizes the formal properties of the spectral sequence.

\begin{prop}
\label{prop:formalpropertieshomologicallerayss}
We consider the following diagram of schemes:
\[
\xymatrix@R=2pt@C=30pt{
& Y\ar_{\pi'}[ld]\ar^{f'}[rd]\ar^p[dd] & \\
S & & B \\
& X\ar^{\pi}[lu]\ar_f[ru] &
}
\]
and a $\T$-spectrum $\E$ over $S$.  Set $\E_X^!=\pi^!\E$ and $\E_Y^!=\pi^{\prime!}(\E)$.  Given any point $x \in B$, let $p_x:Y_x \rightarrow X_x$ be the induced morphism on the fibers over a point $x \in B$.
\begin{itemize}
\item If $p$ is proper, then there is an adjunction map:
\[
\phi_p:f'_*\E_Y^!=f'_*\pi^{\prime!}(\E)=f_*p_!p^!\pi^!(\E)
\xrightarrow{ad(p_!,p^!)} f_*\pi^!(\E)=f_*(\E_X^!)
\]
which induces a morphism of spectral sequences converging to the indicated map on the abutment:
$$
\xymatrix@R=15pt@C=10pt{
A_p^\delta\big(B,H^\delta_{q}(f'_*\E_Y^!))\big)\ar_{\phi_{p*}}[d]
 \ar@{=>}[r] & \E_{p+q}(Y/S)\ar^{p_*}[d] \\
 A_p^\delta\big(B,H^\delta_{q}(f_*\E_X^!))\big)
 \ar@{=>}[r] & \E_{p+q}(X/S).
}
$$
Moreover, the map $\phi_{p*}$ is induced by the following morphism of Gersten complexes:
\[
\xymatrix@R=18pt@C=32pt{
\hdots\ar[r] & \bigoplus_{x \in B_{(p)}} \E_{p+q}(Y_x/S)\ar^-{d_{p,Y}^G}[r]\ar_{\sum_x p_{x*}}[d]
 & \bigoplus_{s \in B_{(p-1)}} \E_{p+q-1}(Y_s/S)\ar[r]\ar^{\sum_s p_{s*}}[d] &  \hdots \\
 \hdots\ar[r] & \bigoplus_{x \in B_{(p)}} \E_{p+q}(X_x/S)\ar^-{d_{p,X}^G}[r]
 & \bigoplus_{s \in B_{(p-1)}} \E_{p+q-1}(X_s/S)\ar[r] &  \hdots
}
\]
\item If $p$ is \'etale, then there is an adjunction map:
\[
\psi_p:f_*(\E_X^!)=f_*\pi^!(\E)
 \xrightarrow{ad(p^*,p_*)} f_*p_*p^*\pi^!(E)=f'_*\pi^{\prime!}(\E)=f'_*\E_Y^!
\]
which induces a morphism of spectral sequences, converging to the indicated map on the abutment:
$$
\xymatrix@R=15pt@C=10pt{
A_p^\delta\big(B,H^\delta_{q}(f_*\E_X^!))\big)\ar_{\psi_{p*}}[d]
 \ar@{=>}[r] & \E_{p+q}(X/S)\ar^{p^*}[d] \\
 A_p^\delta\big(B,H^\delta_{q}(f'_*\E_Y^!))\big)
 \ar@{=>}[r] & \E_{p+q}(Y/S).
}
$$
Moreover, the map $\phi_{p*}$ is induced by the following morphism of Gersten complexes:
\[
\xymatrix@R=18pt@C=32pt{
\hdots\ar[r] & \bigoplus_{x \in B_{(p)}} \E_{p+q}(X_x/S)\ar^-{d_{p,X}^G}[r]\ar_{\sum_x p_{x}^*}[d]
 & \bigoplus_{s \in B_{(p-1)}} \E_{p+q-1}(X_s/S)\ar[r]\ar^{\sum_s p_{s}^*}[d] &  \hdots \\
 \hdots\ar[r] & \bigoplus_{x \in B_{(p)}} \E_{p+q}(Y_x/S)\ar^-{d_{p,Y}^G}[r]
 & \bigoplus_{s \in B_{(p-1)}} \E_{p+q-1}(Y_s/S)\ar[r] &  \hdots
}
\]
\end{itemize}
\end{prop}
\begin{proof}
Each point is obtained by using the functoriality of the homotopy Leray
 spectral sequence with respect to the $\T$-spectrum $\E$. The computation of the map on the abutment
 follows from the definition of the functoriality of the bivariant theory.
 The map on the $E_2$-terms follows from point (2) of Proposition \ref{prop:Gersten}.
\end{proof}

\begin{rem}
Using fundamental classes as defined in \cite{DJK},
 and the induced functoriality on bivariant theory, one can extend the
 contravariant \'etale functoriality to smoothable lci morphisms $f:Y \rightarrow X$,
 up to considering twists by the Thom space of the cotangent complex of $f$.
 We leave the formulation to the reader.
\end{rem}

It is possible to describe the filtration on the abutment of the homotopy Leray
 spectral sequence in geometric terms. The following
 definition is the obvious generalization of the classical definition of Grothendieck
 (see \cite{BO}).
\begin{df}
Consider the setting of Paragraph \ref{num:Leray_ssp}.
 We define the $\delta$-niveau filtration on $\E_*(X/S)$ relative to $f$ as:
$$
{}^\delta N_p^f\E_*(X/S)
=\bigcup_{i:Z \rightarrow X, \delta(Z)\leq p}
 \mathrm{Im}\big(i_*:\E_*(X \times_B Z/S) \rightarrow \E_*(X/S)\big).
$$
where $i$ runs over the closed immersions with target $X$.
\end{df}

\begin{rem}
One can also describe this filtration by the following formula:
$$
{}^\delta N_p^f\E_*(X/S)
=\bigcup_{\pi:Y \rightarrow X, \delta(Y)\leq p}
 \mathrm{Im}\big(\pi_*:\E_*(X \times_B Y/S) \rightarrow \E_*(X/S)\big).
$$
where $\pi:Y \rightarrow X$ runs over the proper morphism with target $X$.
 Indeed, such a morphism always factors as
$$
Y \xrightarrow{\bar \pi} \pi(Y) \xrightarrow i X
$$
where $\pi(Y)$ denotes the image of $Y$, with its canonical structure of a closed subscheme of $X$.
 As the  map $\bar \pi$ is surjective, one has $\delta(\pi(Y)) \leq p$. This concludes.
\end{rem}

\begin{num}
Before stating the computation of the filtration on the $\delta$-homotopy spectral sequence,
 we will introduce notations in order to simplify the proof.
 We consider again the assumptions and notations of Paragraph \ref{num:Leray_ssp}.

First, note that by definition, the filtration induced by the $\delta$-homotopy spectral
 sequence on $\E_*(X/S)$ is defined as:
$$
{}^\delta F_q^f\E_*(X/S)=\mathrm{Im}\big(H_*(B,\tau_{\geq q} f_* \E_X^!) \rightarrow \E_*(X/S)\big).
$$

Second, we can define the $\delta$-niveau filtration relative to $f$ at the level of schemes,
 by considering the following ind-schemes:
$$
B_{\leq p}=\pilim{Z_* \in \flag(B)} Z_p.
$$
Therefore we get a closed ind-immersion $B_{\leq p} \xrightarrow{i_p} B$.
We can define a kind of complementary immersion by considering the following pro-objects:
$$
B_{>p}=\pplim{Z_* \in \flag(B)} (B-Z_p)
$$
together with the pro-open immersion $B_{>p} \xrightarrow{j_p} B$.

Using these notations, we can consider the localization long exact sequence
$$
\E_*(X \times_B B_{\leq p}/S) \xrightarrow{i_{p*}} \E_*(X/S)
 \xrightarrow{j^*_p} \E_*(X \times_B B_{>p}/S) \xrightarrow{\partial_p} \E_*(X \times_B B_{\leq p}/S)
$$
where the third (resp. first, fourth) member(s) is the obvious colimit,
 using the contravariance of $\E_*(-/S)$ with respect to open immersion (resp. contravariance
 with respect to closed immersions). This long exact sequence is nothing else than the filtered
 colimit of the localisation sequences with respect to the closed immersions
 $X \times_B Z_p \rightarrow X$
 for $\delta$-flags $Z_*$ of $B$.

With these notations, the $\delta$-niveau filtration relative to $f$ simply equals the image
 of $i_{p*}$.
\end{num}
\begin{prop}
\label{prop:comparison_filt}
Consider the above assumptions and notation.
 Then for any pair of integer $(p,n) \in \ZZ^2$, one has the following relation:
$$
{}^\delta F_p^f\E_n(X/S)={}^\delta N_{n-p}^f\E_n(X/S),
$$
where the left hand-side is the filtration on the abutment of the homotopy Leray
 spectral sequence associated with $f$ and the right-hand side is the $\delta$-niveau
 filtration relative to $f$.
\end{prop}

\begin{proof}
Let us put $\E'=f_*\E_X^!=f_*\pi^!\E$. We want to compare the following filtrations:
\begin{align*}
{}^\delta F_q^f\E_*(X/S)
 &=\mathrm{Im}\big(H_*(B,\tau_{\geq q} \E') \rightarrow H_*(B,\E')\big), \\
{}^\delta N_p^f\E_*(X/S)
 & =\mathrm{Im}\big(H_*(B_{\leq p},\E') \rightarrow H_*(B,\E')\big).
\end{align*}
So, reasoning with $\E'$ instead of $\E$,
 we reduce to the case $X=B=S$, $f=Id_X$, $\pi=Id_X$.

We will start with the following lemma:
\begin{lm}
Let $\E$ be a $\T$-spectrum over a scheme $X$. Then one has the following vanishing:
\begin{align*}
H_n\big(X_{\leq p},\tau^\delta_{<q}\E\big)=0 & \text{ if } n \geq p+q, \\
H_n\big(X_{>p},\tau^\delta_{\geq q}\E\big)=0 & \text{ if } n \leq p+q.
\end{align*}
\end{lm}
This follows from the case $X=S$, $m=q$, $i=0$ of the equivalent conditions of Paragraph \ref{num:homological_delta_ref}.

We can now consider the following commutative diagram, whose rows and columns
 are exact sequences:
$$
\xymatrix{
& H_{n+1}\big(X_{\leq p},\tau^\delta_{<q}\E\big)\ar[d] & & \\
H_{n+1}\big(X_{>p},\tau^\delta_{\geq q}\E\big)\ar[r]
 & H_{n}\big(X_{\leq p},\tau^\delta_{\geq q}\E\big)\ar_a[d]\ar^b[r]
 & H_{n}\big(X,\tau^\delta_{\geq q}\E\big)\ar^c[d]\ar[r]
 & H_{n}\big(X_{>p},\tau^\delta_{\geq q}\E\big) \\
& H_{n}\big(X_{\leq p},\E\big)\ar[d]\ar_d[r]
 & H_{n}\big(X,\E\big)
 & \\
& H_{n}\big(X_{\leq p}\tau^\delta_{<q},\E\big) & & \\
}
$$
According to the preceding lemma, one gets that:
\begin{itemize}
\item $a$ is an isomorphism if $n \geq p+q$.
\item $b$ is an isomorphism if $n \leq p+q-1$, and an epimorphism if $n=p+q$.
\end{itemize}
Therefore, we obtain that $\mathrm{Im}(c)=\mathrm{Im}(d)$ if $n=p+q$. This concludes.
\end{proof}

\begin{ex}
The preceding proposition recovers the conjecture of Washnitzer proved by Bloch and Ogus
 in \cite[6.9]{BO}. This is obtained as follows:
\begin{itemize}
\item $S$ is the spectrum of a perfect field $k$, $X$ is smooth over $k$;
\item $\E=\E_{dR}$ is the (ring) spectrum representing De Rham cohomology as in \cite[\textsection 3.1]{CD2}
 while $\T=\DM(-,\QQ)$.
\end{itemize}
The fact the truncation of $\E_{dR}$ for the homotopy $t$-structure agrees with the truncation
 of the De Rham complex follows from the construction of $\E_{dR}$ (\cite[3.1.5]{CD2}).
\end{ex}

\begin{rem}
The proof of Bloch and Ogus is less precise and more theoretical. It consists in proving that
the niveau spectral sequence for De Rham cohomology agrees with the hypercohomology spectral sequence
associated with the De Rham complex.

We do not need such a comparison to prove our result, and our proof is more direct.
 But however, let us indicate that there is also an underlying comparison of spectral sequences.
 In fact, it is possible to prove that the homotopy Leray spectral sequence defined above agrees
 from $E_2$-on with the $\delta$-niveau spectral sequence \ref{df:delta_niveau} of $X/S$ with
 coefficients in the spectrum $f_*(\E)$.
The case where $X=B$, $f=Id_B$, $S$ is the spectrum of a perfect field was proved in \cite{Bon1}
 and \cite{Deg13}. The general case will be treated in future work.
\end{rem}

\subsection{The cohomological version}
\label{sec:coh_version}

\begin{num}
\label{num:spectral_diagrams}
Let us consider again the assumptions of Paragraph \ref{num:Leray_ssp}, but with $X=S$.  So we fix a morphism of schemes:
\[
f:X \rightarrow B
\]
and a $\T$-spectrum $\E$ over $X$.

To get products on the $\delta$-homotopy spectral sequence of $(f,\E)$, we will use a classical construction of Douady (see \cite[Section II]{Douady}) which uses the theory of \emph{spectral diagrams} from \cite[XV.7]{CE} rather than that of exact couples.

Here is how one gets such a spectral diagram underlying the spectral sequence \eqref{eq:cohomological_E1_Leray}.
 We first define the following tautological functors in a $\T$-spectrum $\F$ over $B$ (in the end $\F$ plays the role of $f_*\E$):
\begin{align*}
\tau^\delta_{\geq -\infty}(\F)=\F, \ &\tau^\delta_{\geq +\infty}(\F)=0, \\
\tau^\delta_{<-\infty}(\F)=0, \ &\tau^\delta_{<+\infty}(\F)=\F.
\end{align*}
Consider the following poset:
\[
\mathcal P=\{ (p,q) \mid p,q \in \ZZ \cup \{\pm \infty\}, p \leq q\}.
\]
such that  $(p,q) \leq (p',q')$ when $p \leq p'$ and $q \leq q'$.  Then for any $(p,q) \in \mathcal P$, we put:
\[
\tau^\delta_{[p,q[}(\F):=\tau^\delta_{\geq p}\tau^\delta_{<q}(\F).
\]
This defines a contravariant functor
\[
\tau^\delta(\F):\mathcal P \rightarrow \T(X),
 (p,q) \mapsto \tau^\delta_{[p,q[}(\F).
\]

Using the truncation triangles associated with the $\delta$-homotopy $t$-structure, we obtain distinguished triangles for $p \leq q \leq r$:
$$
\tau^\delta_{[q,r[}(\F)
 \rightarrow \tau^\delta_{[p,r[}(\F)
 \rightarrow \tau^\delta_{[p,q[}(\F)
 \xrightarrow{\partial} \tau^\delta_{[q,r[}(\F)[1]
$$
where the first two maps are given by the functoriality of $\tau^\delta$.
These formulas imply that the contravariant functor:
$$
\mathcal P \mapsto \mathscr Ab, (p,q) \mapsto H_*\big(B,\tau^\delta_{[p,q[}(f_*\E)\big)
$$
defines a $\ZZ$-graded spectral diagram in the sense of \emph{loc. cit.}.\footnote{Property (SP.5) of \emph{loc. cit.} is not immediate.  It follows from the convergence of the homotopy Leray spectral sequence
 --- see Theorem \ref{thm:Leray_ssp}.}  Again, this construction is obviously functorial in $\E$.

Note that for $\F=f_*(\E)$, the exact couple of Paragraph \ref{num:Leray_ssp} is contained in the preceding spectral
 diagram --- it corresponds to pairs $(p,p+1)$. Therefore, the spectral sequence associated with the preceding spectral diagram in \cite{CE} coincides with the $\delta$-homotopy Leray spectral sequence of $(f,\E)$.
In particular, we get the following formula:
$$
E^r_{p,q}(f,\E)
 =\mathrm{Im}\Big(H_{p+q}\big(B,\tau^\delta_{[p-r,p[}(f_*\E)\big)
 \rightarrow H_{p+q}\big(B,\tau^\delta_{[p-1,p+r-1[}(f_*\E)\big) \Big).
$$
 \end{num}

\begin{rem}
After the seminal work of Cartan and Eilenberg, spectral diagrams have appeared in other forms in the literature. In the triangulated context, one can refer, mainly for historical purposes, to \cite{Verdier}. The $\infty$-categorical
 context is much more recent, but also much more powerful and satisfactory.  The notion of spectral diagram in the $\infty$-categorical context is introduced by Lurie in \cite[Section 1.2.2]{LurieHA}, under the name of $\ZZ$-complex.
 The advantage of using $\infty$-category with pushouts (eg: stable $\infty$-categories) is that a $\ZZ$-complex is essentially equivalent to a tower of objects: see \cite[Lemma 1.2.2.4]{LurieHA}.  The fact that we can stay in the old-fashioned world of triangulated categories stems from the good behavior of $t$-structures.
\end{rem}

To obtain products on the homotopy Leray spectral sequence following Douady, we need to refine Proposition \ref{prop:homotopy_pairing} as follows.

\begin{prop}
\label{prop:products1}
We consider the preceding notations and assume $\delta \geq 0$.
Suppose we are given a morphism of $\T$-spectra over a scheme $B$:
\[
\mu:\F \otimes \F' \rightarrow \F''.
\]
Then for any triple of integers $(p,q,r)$ such that $r \geq 0$,
 there exists a canonical map:
$$
\tau^\delta_{[p,p+r[}(\F) \otimes \tau^\delta_{[q,q+r[}(\F')
 \xrightarrow{\varphi_{p,q,r}} \tau^\delta_{[p+q,p+q+r[}(\F'').
$$
Moreover, this pairing satisfies the formulas (SPP  1) and (SPP 2) of \cite[II.A, D\'efinition, p. 19-06]{Douady}.\footnote{(SPP1) is the functoriality of this pairing with respect to $(p,p+r)$ (resp. $(q,q+r)$) considered as an object of $\mathcal P$. (SPP2) is the Leibniz rule for the boundary map of type $\partial$.}
\end{prop}

\begin{proof}
It is sufficient to treat the case where $\F''=\F \otimes \F'$ and $\mu$ is the identity.  As by assumption $\delta\geq 0$, the tensor product respects homologically positive objects (see the end of Paragraph \ref{num:basic_t_exact}).

According to Proposition \ref{prop:homotopy_pairing}, we get a canonical pairing:
\[
\tau_{\geq p}^\delta(\E) \otimes \tau^\delta_{\geq q}(\F)
 \rightarrow \tau^\delta_{\geq p+q}(\E \otimes \F).
\]
Let us consider the following diagram:
\[
\xymatrix@R=20pt@C=30pt{
\tau_{\geq p}^\delta(\E) \otimes \tau^\delta_{\geq q+r}(\F)\ar[r]\ar[d]\ar@{-->}[rdd]
 & \tau^\delta_{\geq p+q+r}(\E \otimes \F)\ar[d] \\
\tau_{\geq p}^\delta(\E) \otimes \tau^\delta_{\geq q}(\F)\ar[r]\ar[d]
 & \tau^\delta_{\geq p+q}(\E \otimes \F)\ar[d] \\
\tau^\delta_{\geq p}(\F) \otimes \tau^\delta_{[q,q+r[}(\F')\ar@{-->}[r]
 & \tau^\delta_{[p+q,p+q+r[}(\F \otimes \F')
}
\]
where the solid arrows form a commutative diagram. The two columns of this diagram come from distinguished triangles. First this implies that the slanted dotted arrow is zero.  Second, it implies that there exists a unique dotted horizontal arrow making the bottom square commutative.

Then we consider the following diagram:
\[
\xymatrix@R=20pt@C=30pt{
\tau_{\geq p+r}^\delta(\E) \otimes \tau^\delta_{[q,q+r[}(\F)\ar[d]\ar@{-->}^{(1)}[rd]
 &  \\
\tau_{\geq p}^\delta(\E) \otimes \tau^\delta_{[q,q+r[}(\F)\ar[r]\ar[d]
 & \tau^\delta_{[p+q,p+q+r[}(\E \otimes \F) \\
\tau^\delta_{[p,p+r[}(\F) \otimes \tau^\delta_{[q,q+r[}(\F')\ar@{-->}_{(2)}[ru]
 &
}
\]
As the tensor product is right $t_\delta$-exact, the object $\tau_{\geq p+r}^\delta(\E) \otimes \tau^\delta_{[q,q+r[}(\F)$ is in homological degree $\geq p+q+r$. By definition, the object $\tau^\delta_{[p+q,p+q+r[}(\E \otimes \F)$ is in homological degree $<p+q+r$. So the map labeled $(1)$ must be $0$.
 Therefore, the map $(2)$ must exist and gives the existence of the pairing $\varphi_{p,q,r}$.  The uniqueness of the construction then guarantees Douady's coherence properties (SPP1) and (SPP2).
\end{proof}

\begin{num}
\label{num:spectral_diagrams&products}
Granted Proposition~\ref{prop:products1}, one may now apply the construction of \cite[Th. II]{Douady}.  Going back to the setting of Paragraph \ref{num:spectral_diagrams}, and to our morphism $f:X \rightarrow B$,
 we consider a pairing of $\T$-spectra over $X$:
\[
\mu:\E \otimes \E' \rightarrow \E''.
\]
As $f_*$ is weakly monoidal (left adjoint of a monoidal functor),
 we get a pairing:
\[
\mu:f_*(\E) \otimes f_*(\E') \rightarrow f_*(\E'').
\]
Applying the preceding proposition and Douady's construction, we get a pairing of
 spectral sequences:
\[
E_r^{p,q}(f,\E) \otimes_\ZZ E_r^{s,t}(f,\E')
 \rightarrow E_r^{p+s,q+t}(f,\E'')
\]
such that the differentials $d_r$ satisfy the usual Leibniz rule.
\end{num}

Following standard usage, when considering products, we renumber the $\delta$-homological spectral sequence cohomologically, and use  our cohomological conventions (see in particular Paragraph \ref{num:products_Gesten_hlg}).  We then obtain the following result.  

\begin{thm}
\label{thm:deltahomotopylerayss}
Suppose $f:X \rightarrow B$ is a morphism of schemes, and let $\E$ be a $\T$-spectrum over $X$.  The constructions of Paragraph~\ref{num:spectral_diagrams} and \ref{num:spectral_diagrams&products} yield a convergent spectral sequence of the form:
\[
E_2^{p,q}(f,\E)=A^p_\delta(B,H^q_\delta(f_*\E)) \Rightarrow H^{p+q}(X,\E).
\]
If $\E$ admits a ring structure, then the spectral sequence is equipped with a multiplicative structure; the product on the $E_2$-term is induced by the construction of Example~\ref{ex:products_Gersten&d-homology}.
\end{thm}

Before proceeding to a computation of the above $E_2$-term, let us introduce the following general construction within the six functors formalism satisfied by $\T$:
\begin{prop}
\label{prop:extension_shrick}
Let $f:X \rightarrow S$ be a separated morphism essentially of finite type.

Then there exists a pair of adjoint functors:
$$
f_!:\T(X) \rightarrow \T(S):f^!
$$
such that for any factorisation $X \xrightarrow u X_0 \xrightarrow{\bar f} S$
 where $u$ is pro-\'etale and $\bar f$ is of finite type, one has: $f^!=u^*\bar f^!$.
\end{prop}
\begin{proof}
We write $X$ as a projective limit of a pro-$S$-scheme $(X_i)_{i \in I}$ whose transition morphisms are \'etale affine of finite type and $X_i/S$ is separated of finite type. If we choose an index $i \in I$, considering the induced factorisation
\[
X \xrightarrow{u_i} X_i \xrightarrow{f_i} S
\]
we put $f_!=f_{i!}u_{i\sharp}$, and $f^!=u_i^*f_i^!$.  This definition does not depend on the chosen index $i$. First assume we have two indexes $i, j \in I$ with a map $j \rightarrow i$. Then we get a commutative diagram:
\[
\xymatrix@R=20pt@C=30pt{
X\ar^{u_j}[r]\ar@{=}[d] & X_j\ar^{f_j}[r]\ar|{\varphi_{ij}}[d] & X\ar@{=}[d] \\
X\ar^{u_i}[r] & X_i\ar^{f_i}[r] & X
}
\]
and we get:
\[
f^!=u_i^*f_i^!=u_j^*\varphi_{ij}^*f_i^!=u_j^*\varphi_{ij}^!f_i^!=u_j^*f_j^!.
\]
In general, given two indices $i,j \in I$, there exists an third index $k \in I$ and maps $k \rightarrow i$, $k \rightarrow j$ which gives the canonical identification.

The definition does not depend on the pro-object chosen to present $X/S$.  Indeed, any two such pro-objects are isomorphic. Finally, any factorisation as in the statement of the proposition can be taken as an element of a pro-object
 presenting $X/S$.
\end{proof}

\begin{rem}
\begin{enumerate}
\item A more rigorous approach for the preceding proof will use a limit argument
as in Deligne's construction of the functor $f_!$ when $f$ is separated of finite type.
Details are left to the reader.
\item The procedure described in the previous proposition extends a classical trick used in
 \cite[2.2.12]{BBD}.
\item The extension of the pair of adjoint functors $(f_!,f^!)$ to the case where
 $f:X \rightarrow S$ is essentially of finite type immediately gives an extension
 of the definition of bivariant theories to schemes $X/S$ essentially of finite type.
 It is straightforward to check this extension, under the continuity assumption (T2)
 agrees with that defined in Paragraph \ref{num:extension}.
\end{enumerate}
\end{rem}

\begin{num}\label{num:compute_E1_coh_Leray}
We endeavor to describe the $E_2$-term of the cohomological form of the homotopy Leray spectral sequence associated with a morphism $f:X \rightarrow B$ and a $\T$-spectrum $\E$ over $X$. Recall from Theorem \ref{thm:Leray_ssp} (with $X=S$ and with cohomological conventions) that $E_2^{p,q}(f,\E)$ is the $p$-cohomology of the Gersten $\delta$-cohomological complex $C^*_\delta(B,H^q_\delta f_*\E)$. Moreover, for $p \in \ZZ$, one has:
\begin{equation}\label{eq:cohomological_E1_Leray}
C^p_\delta(B,H^q_\delta f_*\E)=\bigoplus_{x \in B^{(p)}} H^{p+q}(X_x/X,\E),
\end{equation}
where we have used the extension of the bivariant theory defined by $\E$ as explained in the preceding remark.
\end{num}



\begin{prop}
\label{prop:compute_E1_coh_Leray}
Notations and assumptions as in \ref{num:compute_E1_coh_Leray}, the following statements hold.
\begin{enumerate}
\item If $f$ is smooth and $\E=f^*\E_0$ for a $\T$-spectrum $\E_0$ over $B$, then, for any $p \in \ZZ$, one has:
    \[
    C^p_\delta(B,H^q_\delta f_*\E)=\bigoplus_{x \in B^{(p)}} H^{p+q}\big(X_x,\E_x\big)
    \]
    where $\E_x=x^!\E$.  More generally, given any point $x:\spec(K) \rightarrow B$ such that $\delta(x)=-p$, and any integer $r \in \ZZ$, one has (recall Definition \ref{df:delta-hlg} and Paragraph \ref{num:products_Gesten_hlg}):
    \[
    \hat H^{q}_\delta f_*\E_r(x)=H^{p+q}\big(X_x,\E_x\gtw{p+r}\big),
    \]
    again with $X_x=x^{-1} X$ and $E_x=x^! \E$.
\item If $X$ is regular, the fibers of $f$ are regular and $\E$ is $X$-pure (Definition \ref{df:purity}), then for any $p \in \ZZ$, one has:
    \[
    C^p_\delta(B,H^q_\delta f_*\E)
    =\bigoplus_{x \in B^{(p)}} H^{p+q}\big(X_x,\E\dtw{-N(X_x/X)}\big)
    \]
    where $N(X_x/X)$ is the normal bundle of the regular closed immersion $i_x:X_x \rightarrow X$ and we set $\E\dtw{-N(X_x/X)}=i_x^!\E \otimes \Th(-N(X_x/X))$.
\end{enumerate}
\end{prop}

\begin{proof}
Consider the first point. We directly prove the assertion concerning a general point $x:\spec K \rightarrow B$, as it implies the remaining assertion.\footnote{To be clear, recall:
\[
H^{p+q}(X_x/X,\E)=(\hat H^{q}_\delta f_*\E)_{-p}(x).
\]
}
As in the proof of Theorem \ref{thm:Leray_ssp}, consider the following cartesian diagram:
\[
\xymatrix@=18pt{
X_x\ar^{i_x}[r]\ar_{f_x}[d] & X\ar^f[d] \\
\spec(K)\ar^-x[r] & B
}
\]
Since $f$ is smooth, there is a canonical isomorphism of functors: $i_x^!f^*=f_x^*x^!$ (using the notation introduced just before the statement of the proposition for the shriek functors). The following computation then concludes the proof:
\begin{align*}
H^{p+q}(X_x/X,\E)&=\big[\un_{X_x},i_x^!f^*(\E)[p+q]\big]
 =\big[\un_{X_x},f_x^*x^!(\E)[p+q]\big] \\
 &=\big[f_{x\sharp}(\un_x),\E_x[p+q]\big]=H^{p+q}\big(X_x,\E_x\big).
\end{align*}

The second point is a direct consequence of the form \eqref{eq:cohomological_E1_Leray} of the Gersten complex and of Proposition \ref{prop:duality}.
\end{proof}

Again, this spectral sequence has good functoriality properties (e.g., in the (ring) $\T$-spectrum $E$).  The next result summarizes other functoriality properties.

\begin{prop}
\label{prop:coh_leray_pullback}
Consider a commutative diagram:
\[
\xymatrix@=10pt{
Y\ar^q[rr]\ar_g[dr] && X\ar^f[ld] \\
 & B & \\
}
\]
and a ring $\T$-spectrum $\E$. Let us put $\E_Y=q^*\E$.  There exists a morphism of converging spectral sequences:
\[
\xymatrix@=12pt{
E_2^{p,q}(f,\E)=A^p_\delta(B,H^q_\delta(f_*\E))\ar@{=>}[r]\ar_{p^*}[d] & H^{p+q}(X,\E)\ar^{q^*}[d] \\
E_2^{p,q}(f',\E')=A^p_\delta(B,H^q_\delta(g_*\E'))\ar@{=>}[r] & H^{p+q}(Y,\E_Y) \\
}
\]
where $q^*$ is the usual pullback on cohomology.
\end{prop}

The morphism of spectral sequences is simply obtained using the functoriality in $\E$ with respect to the following map:
\[
\E \rightarrow q_*q^*(\E)=q_*\E_Y.
\]

\subsection{Remarkable properties of homotopy modules}
\label{subsec:simplicity}

\begin{num}
\label{num:deltaeffectivity}
\textit{$\delta$-effectivity}.--
Consider the situation of Paragraph~\ref{num:Leray_ssp}: $\E$ is a $\T$-spectrum over $S$ and one looks at morphisms:
\[
\xymatrix@=15pt{
X\ar^f[rr]\ar_\pi[rd] && B \\
& S &
}
\]
where $\pi$ is separated essentially of finite type, and put $\E_X^!=\pi^!\E$. We assume in addition:
\begin{enumerate}
\item $\E_X^!$ and $\un_B$ are $\delta$-effective (Definition \ref{df:delta-eff}).
\item $f$ is proper.
\end{enumerate}
Then the first and second assumptions imply $f_*(\E_X^!)=f_!(\E_X^!)$ is $\delta$-effective
 (see Paragraph \ref{num:delta-effective}). Moreover, one can compute the $E_2$-term of the homotopy Leray spectral sequence as follows:
\begin{align*}
\Hom_{\T(B)}(\un_B[p],H_q^\delta f_*\E)
&=\Hom_{\T(B)}(s\un_B[p],H_q^\delta s f_*\E) \\
&=\Hom_{\T^{\delta\eff}(B)}(\un_B[p],w H_q^\delta s f_*\E) \\
&=\Hom_{\T^{\delta\eff}(B)}(\un_B[p],H_q^{\delta\eff} w s f_*\E) \\
&=\Hom_{\T^{\delta\eff}(B)}(\un_B[p],H_q^{\delta\eff} f_*\E).
\end{align*}
The first identification uses the assumption (1), the second the adjunction $(s,w)$, the third the fact $w$ is $t_\delta$-exact and the last one the fact $s$ is fully faithful.

We have obtained the following remarkable result.
\end{num}

\begin{prop}
With the assumptions of \ref{num:deltaeffectivity} in place, the homotopy Leray spectral sequence takes the following form:
\[
E^2_{p,q}=A_p^\delta(B,H_q^{\delta\eff}f_*\E_X^!) \Rightarrow \E_{p+q}(X/S).
\]
If $X=S$, then we can also consider the cohomological form of the $\delta$-homotopy Leray spectral sequence:
\[
E_2^{p,q}=A^p_\delta(B,H^q_{\delta\eff}f_*\E) \Rightarrow \E^{p+q}(X).
\]
\end{prop}

\begin{ex}
The interest of the preceding proposition is that it is easier to get bounded objects with respect to the $\delta$-effective category: see Example \ref{ex:unit_eff}.

Let us consider either the homotopical or motivic case.  We consider a proper morphism $f:X \rightarrow B$ of schemes essentially of finite type over $k$ such that $X$ is $k$-smooth.  We let our dimension functions be induced by that computed relative to $k$.  Then $\un_X(n)$ for $n\geq 0$ and $\un_B$ are both $\delta$-positive.

If we assume $X$ is of pure dimension $d$, so that $\delta(X)=d$ according to the preceding choice, then the constant object $\un_X$ is concentrated in homological degree $d$. Therefore the preceding spectral sequence, in its homological form, is concentrated in degrees $q \in [d,d+\dim(f)]$.
\end{ex}

The following definition is an analog of the notion of local system in classical topology.
%

\begin{df}
\label{df:simplehomotopymodule}
Suppose $p:B \rightarrow S$ is a separated morphism of finite type and $\E$ is a $\T$-spectrum over $B$.  We will say that $\E$ is \emph{$S$-simple} if there exists a $\T$-spectrum $\E_0$ over $S$ and an isomorphism $\E \simeq p^!\E_0$.  We will say that $\T$ is \emph{locally simple over $S$} (or locally $S$-simple)
 if there exists a Nisnevich cover $\pi:W \rightarrow B$ such that $\pi^*\E$ is $S$-simple (with respect to the projection $p \circ \pi$).
\end{df}

Note that the terminology $S$-simple is analogous to the classical terminology of "simple local system". This corresponds to the case of trivial monodromy.

\begin{rem}
This notion will come into play mainly when $S$ is the spectrum of a base field $k$, or, in the motivic case, of a base Dedekind ring $A$.  For us, the interest comes, as in topology, in the study of the homotopy Leray spectral sequence (Definition \ref{df:Leray_ssp}). Usually, we will start with a $k$-simple $\T$-spectrum $\E$ over $S$ -- thus $\E_X$ is $k$-simple.  If we know that the homotopy module $H_q^\delta  f_*\E$ over $B$ is $k$-simple, then the $E_2$-term depends only on the homotopy type (resp. motive) of $B$ over $k$.  We will give examples in Section \ref{sec:applications}.
\end{rem}

\begin{rem}
\begin{enumerate}
\item Obviously, locally $S$-simple $\T$-spectra over $B$ are stable under suspensions, twists
 and even tensor products by Thom spaces of virtual bundles over $B$.
  The same is true for $S$-simple $\T$-spectra except
   that one can only twist  them by virtual bundles over $B$
    that come from $S$.
 Neither of these notions is stable under extensions or even direct factors in general.
\item Let $f:B \rightarrow S$ be smooth and consider the motivic case.
 Then $f^!=f^*(d)[2d]$ where $d$ is the relative dimension of $f$.
 One deduces that $S$-simple motives over $B$ are stable under tensor products.
 The same remark applies to oriented spectra over $B$,
 but not to arbitrary spectra.
\item Note that if $k$ is a field, $k$-simple over a scheme $B$ implies $B$-pure
 (Definition \ref{df:purity}).
\item One could say that a homotopy module over $B$ is $S$-constructible
 if it is obtained by a finite number of extensions and direct factors  of locally $S$-simple homotopy modules,
 within the abelian category of homotopy modules over $B$.

This notion is not so well-behaved, compared to its model for torsion \'etale sheaves,
 as it lacks some notion of finiteness. It would be desirable to have some good finiteness
 condition on $S$-simple homotopy modules. But even when $S$ is a base field,
 it is not obvious to find such a finiteness condition; see \cite[Rem. 6.7]{Deg9} for
 further discussion.
\end{enumerate}
\end{rem}

Note the following fact.
\begin{lm}
Consider a separated morphism $f:B \rightarrow S$ of finite type,
 and a $\T$-spectrum $\E$ over $B$.
 Let us fix unrelated dimension functions $\delta$ on $B$
  and $\delta_0$ on $S$.

If $\E$ is $S$-simple, then for any $p \in \ZZ$, the homotopy module
 $H_q^\delta(\E)$ over $B$ is $S$-simple.
\end{lm}

\begin{proof}
We just need to be precise about dimension functions. By additivity, we can assume that $B$
 is connected.
 Then the dimension function $\delta_0^f$ on $B$,
  induced by $\delta_0$, satisfies the relation
 $\delta_0^f=\delta+n$ for a fixed integer $n \in \ZZ$ (see Remark \ref{rem:delta_independence}).
 In particular, $H_q^{\delta_0}(\E)=H_{q+n}^\delta(\E)$ according to \emph{loc. cit.}

By assumption, $\E=f^!\E_0$. it remains to apply the fact $f^!$ is $t_{\delta_0}$-exact to conclude:
\[
H_{q}^{\delta}(\E)=H_{q-n}^{\delta_0}(\E) \simeq H_{q-n}^{\delta_0}(f^!\E_0)
 =f^!H_{q-n}^{\delta_0}(\E_0).
\]
\end{proof}

\begin{ex}
In the motivic case, if $f:X \rightarrow S$ is smooth, then the constant object $\un_X$ is $S$-simple.  In contrast, if  $f$ has sufficiently complicated singularities, then $\un_X$ may fail to be $S$-simple.  


Similarly, the classical oriented ring spectra $\HH R_X$, $\KGL_X$ (K-theory), $\MGL_X$ (algebraic cobordism), $\hat \HH \QQ_{\ell,X}$ (representing continuous $\ell$-adic cohomology) are all $S$-simple.

On the contrary, spectra representing non-orientable theories such as $\HH \tilde R_X$, $\mathrm{KQ}_X$ (hermitian K-theory), or the sphere spectrum $\mathbf S^0_X$, are not $S$-simple except when the tangent bundle of $f$ is trivial (or is the pullback of a vector bundle over $S$).  In any case, they all are locally $S$-simple.

Note finally that when $f:X \rightarrow S$ is arbitrary separated of finite type, the main result of \cite{JinGth}, Theorem 1.3, tells us that the spectrum $\mathbf{GGL}_X$ representing algebraic $G$-theory in $\SH(X)$ is $S$-simple.
\end{ex}

As expected, here is the generic case where the homotopy modules appearing in the $E_2$-term of the homotopy Leray spectral sequence are simple.
\begin{prop}
\label{prop:trivial_fibration}
Let $B$ and $F$ be $S$-schemes separated essentially of finite type.  We consider the trivial fibration $f:X=F \times_S B \rightarrow B$.  Then for any $S$-simple $\T$-spectrum $\E$ over $X$, the $\T$-spectrum $f_*\E$ is $S$-simple.
\end{prop}

\begin{proof} This is a trivial exercise on the six functors formalism.  We consider the cartesian square:
\[
\xymatrix@=20pt{
X\ar^f[r]\ar_q[d]\ar@{}|\Delta[rd] & B\ar^p[d] \\
F\ar^{f_0}[r] & S.
}
\]
Then we get an associated exchange isomorphism: $Ex(\Delta^!_*):p^!f_{0*} \xrightarrow \sim f_*q^!$.  By assumption, there exists a $\T$-spectrum $\E_0$ over $S$ such that $\E=h^!\E_0$,
 $h=pf=f_0q$.  Thus we can do the computation:
\[
f_*\E=f_*h^!\E_0=f_*q^!f_0^!\E_0 \xrightarrow{Ex(\Delta^!_*)^{-1}} p^!f_{0*}f_0^!\E_0.
\]
This concludes the proof.
\end{proof}

\begin{cor}
Let $f:X \rightarrow B$ be Nisnevich locally principle fibration
 of $S$-schemes. Then for any locally $S$-simple $\T$-spectrum $\E$ over $B$,
 the $\T$-spectrum $f_*\E$ is locally $S$-simple.
\end{cor}

\section{Applications}
\label{sec:applications}
In this section, we study the homotopy Leray spectral sequence in various simple cases and include some applications.  Section~\ref{subsec:aonecontractiblefibers} studies morphism with (stably) $\AA^1$-contractible fibers, and then morphisms with fibers that are ``$\AA^1$-homology spheres".  Section~\ref{subsec:gysinwang} studies ``fibrations" with either base or fiber that are motivic spheres.  Finally, Section~\ref{subsec:relativecellular} is concerned with some applications of the spectral sequence to relative cellular spaces.

\subsection{Morphisms with $\AA^1$-contractible fibers}
\label{subsec:aonecontractiblefibers}
We first analyze the case where the fibers are $\AA^1$-homotopically ``as simple as possible", i.e., $\AA^1$-contractible.

\begin{prop}
\label{prop:degenerationwithcontractiblefibers}
Suppose $f:X \rightarrow B$ be a smooth morphism with $\AA^1$-contractible fibers.  Let $M$ be a spectrum (resp. a motive) that is a homotopy module over $B$ and set $M_X := f^*M$.  Then,
\[
H^q_\delta f_*M_X=\begin{cases}
M & q=0, \\
0 & q\neq 0.
\end{cases}
\]
In particular, the homotopy Leray spectral sequence is concentrated on the line $q=0$, and thus degenerates; pullback along $f$ yields identifications (in cohomological notation) of the form:
\[
H^p(X,M_X)=H^p(B,H^0_\delta f_*(M_X))=H^p(B,M).
\]
\end{prop}
Note in particular that $f_*$ respects $S$-simple and locally $S$-constant objects.
\begin{proof}
Consider the map
\[
\alpha:H^q_\delta(M) \rightarrow H^q_\delta(f_*M_X)
\]
induced by the adjunction map $M \rightarrow f_*f^*(M)$.
We need only to prove it is an isomorphism on fiber homology (Remark \ref{rem:Rost_modules}). We compute the above map,
 evatuated at a point $x:\spec K \rightarrow B$ and in $\GG$-degree $r \in \ZZ$:
\[
H^q_\delta(M)_r(x) \rightarrow H^q_\delta(f_*M_X)_r(x).
\]
We use the computation of Proposition~\ref{prop:compute_E1_coh_Leray}(1),
 which we can apply as $f$ is smooth.
 So if we put $\delta(x)=-p$ and $M_x=x^!M$
 (using the notation of Proposition \ref{prop:extension_shrick}), one obtains:
\begin{align*}
H^q_\delta(M_X)_r(x)&=H^{p+q}(x,M_x\gtw{p+r}), \\
H^q_\delta(f_*M_X)_r(x)&=H^{p+q}(X_x,M_x\gtw{p+r}).
\end{align*}
Moreover, the canonical map $\alpha$ is isomorphic to the pullback map:
\[
f_x^*:H^{q+r-2p,r-p}(x,M_x) \rightarrow H^{q+r-2p,r-p}(X_x,M_x),
\]
which is itself an isomorphism as $f_x$ is assumed to be an $\AA^1$-weak equivalence.  Note also that, because the spectral sequence is functorial with respect to pullbacks (Proposition \ref{prop:coh_leray_pullback}), we know that the identification of the statement arises from the pullback map along $f$.  The remaining assertions are straightforward.
\end{proof}

\begin{ex}
There exist many examples of morphisms satisfying the hypotheses of Proposition~\ref{prop:degenerationwithcontractiblefibers}.  For concreteness, assume $k$ is a field having characteristic $0$.  In \cite[Theorem 1.3]{ADContractible}, it is shown that there exist connected $k$-schemes $B$ of arbitrary dimension and smooth morphisms $f: X \to B$ of relative dimension $\geq 6$ whose fibers are ${\mathbb A}^1$-contractible and such that fibers over distinct $k$-points of $B$ are pairwise non-isomorphic.  Because of the last point, such morphisms $f$ are {\em not} Zariski locally trivial.  These results were improved in \cite{DuboulozFasel,DuboulozPauliOstvaer}, where it was shown that one could build $f$ as above that are smooth of relative dimension $\geq 3$.  On the other hand, it is expected that the only $\AA^1$-contractible smooth $k$-scheme of dimension $2$ is $\AA^2$, and a long-standing conjecture of Dolgachev--Weisfeiler \cite[3.8.5]{DolgachevWeisfeiler} states that every flat morphism of (say) smooth schemes with all fibers isomorphic to affine space is Zariski locally trivial.

Suppose $f$ is a smooth morphism $f$ with $\AA^1$-contractible fibers.  It is not clear to the authors whether such an $f$ is {\em unstably} an $\AA^1$-weak equivalence without imposing further hypotheses (e.g., that $f$ is Nisnevich locally trivial).  Nevertheless, the following remark demonstrates that such $f$ are {\em stable} $\AA^1$-weak equivalences in a strong sense, which makes Proposition~\ref{prop:degenerationwithcontractiblefibers} somewhat unsurprising.
\end{ex}

\begin{rem}
\label{rem:stable_A1_w_eq}
Given a map $f:X \rightarrow B$ as in Proposition~\ref{prop:degenerationwithcontractiblefibers}, one can directly show that the adjunction map:
\[
Id \rightarrow f_*f^*
\]
is an isomorphism of functors. Indeed, we can use the same argument as above and the fact the family of functors $x^!:\SH(B) \rightarrow \SH(\spec(\kappa(x)))$, indexed by schematic points $x \in B$, is conservative.\footnote{One uses for example Proposition
\ref{prop:extension_shrick} to define $x^!$.  The conservativity property is obtained using the continuity property of $\SH$ and the localization property.}  In particular, one does not need to work over a base field; nor does one need to invert any integers.

As $f$ is smooth, we also get a natural transformation $f_\sharp f^* \rightarrow Id$,
 which is an isomorphism according to the result of the preceding paragraph.
 In particular, the map
\[
\Sigma^\infty X_+=f_\sharp f^*(\un_B) \rightarrow \un_B=\Sigma^\infty B_+
\]
is an isomorphism in $\SH(B)$.  In fact, $f$ is a {\em universal} stable $\AA^1$-weak equivalence since $f$ is a stable $\AA^1$-weak equivalence and the property mentioned above remains true after base change.
\end{rem}


\subsection{Gysin and Wang sequences}
\label{subsec:gysinwang}
The relative Atiyah--Hirzebruch spectral sequence takes a particularly simple form when the Serre fibration $F \to X \to B$ under consideration has either the property that $B$ is a sphere, or $F$ is a homology sphere and the associated local system on $B$ is trivial.  In those cases, the spectral sequences yields the so-called Wang or Gysin long exact sequences.  The fact that the differentials in the cohomological form of the spectral sequence are derivations yields additional structure in these long exact sequences that is frequently useful in computations.

\subsubsection*{${\mathbb A}^1$-homology spheres}
In motivic homotopy theory, there are many smooth schemes over a base $S$ that have the stable $\AA^1$-homotopy type of a motivic sphere $\Sigma^i \mathbf{G}_m^{\wedge j}$.  For example, Morel--Voevodsky showed that $\AA^n \setminus 0_S$ has the $\AA^1$-homotopy type of $\Sigma^{n-1}\mathbf{G}_m^{\wedge n}$.  Likewise, the split smooth affine quadric $Q_{2n-1}$ defined by the hypersurface $\sum_i x_ix_{n+i} = 1$ in $\AA^{2n}_S$ is $\AA^1$-weakly equivalent to $\AA^n \setminus 0$ and \cite[Theorem 2]{AsokDoranFasel} demonstrates that the smooth affine quadric $Q_{2n}$ defined by the equation $\sum_i x_i x_{n+i} = x_{2n+1}(1 - x_{2n+1})$ in $\AA^{2n+1}_{S}$ is a model of $\Sigma^n \mathbf{G}_m^{\wedge n}$.  On the other hand, it is known that $\Sigma^i \mathbf{G}_m^{\wedge j}$ has no smooth model if $i > j$ and conjecturally has no smooth model if $i < j-1$.  We now formulate a definition of ``homology sphere" in $\AA^1$-homotopy theory.

\begin{df}
\label{df:aonehomologysphere}
We say that a (pointed) smooth $S$-scheme $X$ is an $\AA^1$-homology sphere if there exist integers $p,q,r \geq 0$ and an $\AA^1$-weak equivalence $\Sigma^r X \sim \Sigma^p \mathbf{G}_m^{\wedge q}$.
\end{df}

The next proposition gives a construction of many $\AA^1$-homology spheres, at least over a field.

\begin{prop}
\label{prop:homologyspheres}
Fix a smooth base scheme $S$, and an $\AA^1$-contractible smooth $S$-scheme $X$.  Assume there exists a closed immersion of $S$-schemes $x: S \to X$ with trivial normal bundle $\nu_{x/X}$.
\begin{enumerate}
\item A choice of trivialization of $\nu_{x/X}$ determines an $\AA^1$-weak equivalence $\Sigma X \setminus x \cong \Sigma \AA^d \setminus 0$, i.e., $X \setminus x$ is an $\AA^1$-homology sphere.
\item If $S = \Spec k$ for a perfect field $k$, $X$ has dimension $d \geq 3$ and is $\AA^1$-connected, then $X \setminus x$ is $\AA^1$-simply connected as well.
\end{enumerate}
\end{prop}

\begin{proof}
For the first point, note that there is a cofiber sequence of the form
\[
X \setminus x \longrightarrow X \longrightarrow X/(X \setminus x) \longrightarrow \Sigma X \setminus x \longrightarrow \cdots
\]
Because the $\AA^1$-local model structure is left proper, the fact that $X$ is $\AA^1$-contractible implies that the map $Th(\nu_{x/X}) \to \Sigma X \setminus x$ is an $\aone$-weak equivalence.  Under the assumptions on $S$, there is a homotopy purity isomorphism $X/(X \setminus x) \cong Th(\nu_{x/X})$ and a choice of trivialization of $\nu_{x/X}$ determines an $\aone$-weak equivalence $Th(\nu_{x/X}) \stackrel{\sim}{\to} \Sigma^d \mathbf{G}_m^{\wedge d}$, which can be written $\Sigma \AA^d \setminus 0$.

For the second point, since $X \setminus x$ is $\AA^1$-connected, $X \setminus x$ has a non-empty set of $k$-points by the unstable $0$-connectivity theorem \cite[\S 2 Corollary 3.22]{MV}.  Fix a base $k$-point in $X \setminus x$, and point $X$ by its composite with the open immersion $X \setminus x \to X$.  Finally, since $d \geq 3$, we may appeal to \cite[Theorem 4.1]{ADExcision} to conclude that the morphism $\pi_1^{\AA^1}(X \setminus x) \to \pi_1^{\AA^1}(X)$ is an isomorphism.  Since $X$ is $\AA^1$-contractible, the latter sheaf is trivial and $X \setminus x$ is thus $\AA^1$-simply connected.
\end{proof}

\begin{rem}
If $k$ is a field, and $X$ is furthermore affine, then $X \setminus x$ is isomorphic to $\AA^d \setminus 0$ if and only if $X$ is isomorphic to $\AA^d$.  Indeed, if there exists an isomorphism from $X \setminus x$ to $\AA^d \setminus 0$, then normality of $X$ allows one to extend this isomorphism to an isomorphism of $X$ with $\AA^d$; the other implication is immediate.
\end{rem}

\begin{ex}
If $X$ is the smooth affine threefold defined by $x + x^2y + z^2 + t^3 = 0$, then the main result of \cite{DuboulozFasel} implies that $X$ is $\AA^1$-contractible, at least if $k$ is an infinite field.  In that case, for any extension $L/k$, B. Antieau observed that $X$ is connected by chains of affine lines (see \cite[Example 2.28]{DuboulozPauliOstvaer} for a proof).  If $L$ is an infinite field, we may always assume our chains avoid a codimension $\geq 2$ subset, and in particular it follows that $X \setminus x$ is connected by chains of affine lines; it follows that, $X \setminus x$ is $\AA^1$-connected.  Thus, if $k$ is infinite and perfect, $X \setminus x$ satisfies the hypotheses of the proposition and yields an ``exotic motivic sphere".  In dimension $d \geq 4$, the examples in \cite{ADContractible} or \cite{AsokDoranFasel} also satisfy the hypotheses of the theorem, at least over an infinite base field.
\end{ex}


\subsubsection*{Gysin and Wang sequences for homotopy modules}
Our goal now is to analyze the homotopy Leray spectral sequence for $f: X \to B$ a smooth morphism where $B$ is an $\AA^1$-homology sphere in the sense above; the outcome will be a version of the Gysin sequence.  We begin by observing that from the appropriate cohomological standpoint, $\AA^1$-homology spheres behave in a fashion analogous to spheres in classical homotopy theory, i.e., cohomology with ``locally constant coefficients" is concentrated in precisely $2$ degrees.  To make this precise, assume we work over a field, and let $\delta$ be the usual dimension function relative to $k$.
We consider homotopy modules $M$ over $k$ such that $f_*M$ is {\em $k$-simple} in the sense of Definition~\ref{df:simplehomotopymodule}.

\begin{lm}
\label{lm:vanishing}
Assume $k$ is a field, and $X$ is a (pointed) smooth $k$-scheme that is an $\AA^1$-homology sphere as in Definition \ref{df:aonehomologysphere}, with $p,q \geq 1$.
 Let $f:X \rightarrow \spec(k)$ be the canonical projection. For any homotopy module $M$ over $k$, $M_X:=f^*M$,
\[
H^i_{\delta}(f_*M_X) = \begin{cases}M &\text{if } i = 0 \\ M_{-q} &\text{ if } i = p-r \\ 0 & \text{ otherwise.}\end{cases}
\]
\end{lm}

\begin{proof}
This result follows essentially immediately from \cite[Lemma 4.5]{AsokFaselSpheres}.  However, to keep the presentation self-contained, we sketch a proof in the spirit of this paper: one proceeds along the same lines as the proof of Proposition~\ref{prop:degenerationwithcontractiblefibers} and appeals to Proposition~\ref{prop:compute_E1_coh_Leray}(1).
Again we need only to compute fiber homology of $f_*(M_X)$.  To this end, take a pair $(x,n)$ where $x:\spec(K) \rightarrow \spec(k)$ is a induced by a field extension of finite type, and $n \in \ZZ$.  To simplify the notation, let us assume that $\delta(x)=0$; set $M_x=x^!M$.
 
We may then appeal to the computation of Proposition~\ref{prop:compute_E1_coh_Leray}(1):
\[
H^i_{\delta}(f_*M_X)_n(x) \cong H^{i}(X_x,M_x\{r\})=H^{i}(K,M_x\{n\})
 \oplus \tilde H^{i}(X_x,M_x\{n\}).
\]
As $X/k$ is pointed, $X_x/K$ is also pointed and we obtain an identification:
\[
H^{i}(X_x,M_x\{n\})=H^{i}(K,M_x\{n\}) \oplus \tilde H^{i}(X_x,M_x\{n\})
\]
where $\tilde H^*$ stands for the reduced cohomology of a pointed scheme.  Furthermore there are isomorphisms of the form:
\[
\tilde H^{i}(X_x,M_x\{n\})
 \cong \tilde{H}^{i+r}(\Sigma^r X_x,M_x)
 \cong \tilde{H}^{i+r}(\Sigma^p \mathbf{G}_{m}^{\wedge q},M_x)
 \cong \tilde H^{i+r-p}(\mathbf{G}_{m}^{\wedge q},M_x).
\]
Since $M_x$ is a homotopy module over $K$, the last group vanishes if $i+r-p$ is not equal to $0$ and is precisely $M_{-q}(\Spec k)$ if $i= p-r$.  The result then follows by unwinding the definitions.
\end{proof}

\begin{prop}
\label{prop:gysinsequenceI}
Assume $k$ is a field and $f: X \to B$ is a Zariski locally trivial smooth morphism of $k$-varieties with $B$ connected and where the fibers $F$ of $f$ are $\AA^1$-homology spheres (where $\Sigma^r F \sim \Sigma^p \mathbf{G}_m^{\wedge q}$ with $p,q \geq 1$).  Assume furthermore that $f$ trivializes on a Zariski open cover $\mathcal{U} = \{ U_i \}_{i \in I}$ of $B$, and that we may fix $x \in \cap_{i \in I} U_i(k)$ (i.e., the intersection is non-empty).  If $M$ is a homotopy module and $H^{j}_{\delta}(f_*M_X)$ is $k$-simple for each $i \geq 0$, then
\[
H^j_{\delta}(f_*M_X) = \begin{cases}M_B & \text{ if } j = 0 \\ (M_{-q})_B & \text{ if } j = p-r \\ 0 & \text{otherwise}.
\end{cases},
\]
and there is a long exact ``Gysin" sequence of the form:
\[
\cdots \longrightarrow H^i(B,M) \longrightarrow H^i(X,M) \longrightarrow H^{i + r - p}(B,M_{-q}) \stackrel{\partial}{\longrightarrow} H^{i+1}(B,M) \longrightarrow \cdots.
\]
\end{prop}

\begin{proof}
By Theorem~\ref{thm:deltahomotopylerayss}, there is a spectral sequence with
\[
E^{i,j}_{2} = A^i(B,H^j_{\delta}(f_*M_X)) \Longrightarrow H^{i+j}(X,M).
\]
Granted the first statement, the spectral sequence is concentrated in two rows, and the existence of the resulting long exact sequence is immediate.

Thus, it remains to prove the first statement.  If $B$ is a field, the computation of $H^j_{\delta}(f_*M_X)$ is simply Lemma~\ref{lm:vanishing}.  The general case reduces to that one using the simplicity assumption and a Zariski patching argument.  In more detail, by assumption, we may choose a $k$-point $x: \Spec k \to B$.  By Lemma~\ref{lm:vanishing} there are induced morphisms $M \to x^*H^0_{\delta}(f_*M_X)$ and $M_{-q} \to x^*H^{p-r}_{\delta}(f_*M_X)$.  Since $H^j_{\delta}(f_*M_X)$ is $k$-simple by assumption, pulling back these morphisms along the structure map $B \to \Spec k$ yields morphisms $M_B \to H^0_{\delta}(f_*M_X)$ and $(M_{-q})_B \to H^{p-r}_{\delta}(f_*M_X)$; moreover, these maps are compatible with restrictions to open sets $U \subset B$.  The vanishing statement for $H^j_{\delta}(f_*M_X)$ will be established in an identical manner.  Since $f$ is Zariski locally trivial, upon fixing a trivialization over the open cover $\mathcal{U}$, appeal to Proposition~\ref{prop:trivial_fibration} yields isomorphisms $M_{U_i} \stackrel{\sim}{\to} H^0_{\delta}(f_*M_X)|_{U_i}$ for any $i \in I$, and similarly isomorphisms of the form $(M_{-q})_{U_i} \stackrel{\sim}{\to} H^{p-r}_{\delta}(f_*M_X)|_{U_i}$.  To check the maps of sheaves on $B$ are isomorphisms, it suffices to check that the isomorphisms upon restriction to $U_i$ just described are compatible on $2$-fold intersections.  However, since the all the sheaves in question are $k$-simple, this compatibility may be checked after pullback along $x$, where it is immediate.
\end{proof}

\begin{rem}
Proposition~\ref{prop:gysinsequenceI} avoids the ``degenerate" case where $f: X \to B$ is a $\mathbf{G}_m$-torsor.  Such morphisms are, of course, always Zariski locally trivial, but behave like covering spaces in topology.  Indeed, under the additional hypotheses in the statement, the sheaf $H^j_{\delta}(f_*M_X)$ is only non-vanishing when $j = 0$, in which case it is isomorphic to $M_B \oplus (M_{-1})_B$.   
\end{rem}

Proposition~\ref{prop:gysinsequenceI} admits a refinement when $M$ is a homotopy module that admits a ring structure (e.g., $K^M_*$ or $K^{MW}_*$).  In such cases, $H^0(B,M)$ is a ring, and we may fix a generator $\xi$ of $H^{0}(B,f_*M_{X,q}) \cong H^0(B,M)$ as a module over this ring.  The class $\partial(\xi)$ then determines an element of $H^{p-r+1}(B,M_{q})$ that we will refer to as the Euler class of $f$. 

\begin{thm}
\label{thm:gysinsequenceII}
Assume $k$ is a field and $f: X \to B$ is a Zariski locally trivial smooth morphism of $k$-varieties with $B$ connected and where the fibers $F$ of $f$ are $\AA^1$-homology spheres (where $\Sigma^r F \sim \Sigma^p \mathbf{G}_m^{\wedge q}$ with $p,q \geq 1$).  Assume furthermore that $f$ trivializes on a Zariski open cover $\mathcal{U} = \{ U_i \}_{i \in I}$ of $B$, and that we may fix $x \in \cap_{i \in I} U_i(k)$ (i.e., the intersection is non-empty).  Suppose $M$ is a homotopy module that admits a ring structure.  If $H^{j}_{\delta}(f_*M_X)$ is $k$-simple for each $i \geq 0$, then a choice of generator $\xi$ of $H^{0}(B,f_*M_{X,q}) \cong H^0(B,M)$ determines an Euler class $e(f) \in H^{p-r+1}(B,M_{q})$, and the exact sequence of Proposition~\ref{prop:gysinsequenceI} takes the form:
\[
\cdots \longrightarrow H^i(B,M) \longrightarrow H^i(X,M) \longrightarrow H^{i + r - p}(B,M_{-q}) \stackrel{\cdot e(f)}{\longrightarrow} H^{i+1}(B,M) \longrightarrow \cdots,
\]
where the connecting homomorphism is given by product with $e(f)$ arising from the ring structure.
\end{thm}

\begin{proof}
Granted the identifications mentioned before the statement, the result follows from the existence of multiplicative structure in the homotopy Leray spectral sequence as discussed in Paragraph \ref{num:spectral_diagrams&products}.
\end{proof}

\begin{prop}
Assume $k$ is a field, and $f: X \to B$ is a smooth morphism where $B$ is an $\AA^1$-homology sphere (where $\Sigma^r X \sim \Sigma^p \mathbf{G}_m^{\wedge q}$ with $p,q \geq 1$).  If $M$ is a homotopy module such that $H^j_{\delta}(f_*M)$ is $k$-simple for each $j \geq 0$, then there is a long exact ``Wang" sequence of the form:
\[
\cdots \longrightarrow H^{i}(X,M) \longrightarrow H^{i}_{\delta}(f_*M)(k) \stackrel{\theta}{\longrightarrow} H^{i+1+r-p}_{\delta}(f_*M)_{-p+r}(k) \longrightarrow H^{i+1}(X,M) \longrightarrow \cdots,
\]
where the map $\theta$ is a graded derivation.
\end{prop}

\begin{proof}
This result follows immediately by combining Theorem~\ref{thm:deltahomotopylerayss} and Lemma~\ref{lm:vanishing}, which shows that the resulting spectral sequence is concentrated in two columns.
\end{proof}




\subsection{Relative cellular spaces}
\label{subsec:relativecellular}
The notion of ``algebraic cell decomposition" and the related notion of ``relative algebraic cell decomposition" has a long history.  E.g., if a split torus acts on a smooth projective scheme $X$ over a field, then Bialynicki--Birula \cite{BB} showed that $X$ may be decomposed as a disjoint union of smooth varieties that are total spaces of vector bundles over connected components of the fixed point loci (which are necessarily smooth).  The cohomological consequences of the existence of such filtrations were observed almost immediately (e.g., one immediately computes Chow groups for smooth projective varieties equipped with a torus action with isolated fixed points).

Karpenko was one of the first to exploit the existence of such algebro-geometric cell decompositions to produce motivic decompositions of varieties (\cite{Karpenko}), in his case absolute Chow motives (over a base field).  Such results have been developed in numerous directions, but we will mainly be concerned with the relative version studied in \cite[8.4.2]{MNP}.  We begin by introducing our own version of ``cellularity''.

\begin{df}
\label{df:cellularity}
Suppose $f:X \rightarrow B$ be a morphism of schemes.  Say that $f$ (or $X/B$) admits a flat (resp. lci) cellular structure if there exists a filtration:
\[
\emptyset=X_{-1} \subset X_0 \subset \hdots \subset X_n \subset X_{n+1}=X
\]
of closed subschemes of $X$ such that for all $\alpha$, the following conditions are satisfied:
\begin{enumerate}
\item[(a)] $f_\alpha=f|_{X_\alpha}:X_\alpha \rightarrow B$ is flat
 (resp. lci);
\item[(b)] setting $U_\alpha := X_\alpha-X_{\alpha-1}$, the restriction
 $p_\alpha=f|_{U\alpha}:U_\alpha \rightarrow S$ is a smooth morphism and stable $\AA^1$-weak equivalence (eg: see Remark \ref{rem:stable_A1_w_eq}).
\end{enumerate}
\end{df}

\begin{rem}
In Karpenko's definition, the morphisms $p_{\alpha}$ are assumed to be vector bundles, which is sufficient for his purposes since he essentially uses the method of Bialynicki-Birula, where the geometric decomposition arises from the action of a split torus on a smooth projective variety.  In contrast, our definition allows, e.g., cellular spaces where the ``cells" are merely $\AA^1$-contractible and we allow ourselves to consider varieties that are not necessarily projective.  Examples show that this additional generality is natural.  For example, the smooth affine quadrics $Q_{2n}$ of dimension $2n$ (the hypersurface in affine space defined by the equation $\sum_i x_iy_i = z(1-z)$) admit a decomposition as an affine space ($x_1 = \cdots = x_n = z = 0$) of dimension $n$, and an open complement that is a strictly quasi-affine $\AA^1$-contractible scheme \cite[Theorems 2.2.5 and 3.1.1]{AsokDoranFasel} (in particular the complement is an $\AA^1$-contractible variety that is not isomorphic to affine space).  Likewise, the Panin--Walter model of quaternionic projective space $\mathrm{HP}^n$ admits a ``cell decomposition" where the cells are strictly quasi-affine $\AA^1$-contractible schemes \cite[Theorem 3.1]{PaninWalterPontryaginClasses}.
\end{rem}

Since the morphism $p_\alpha$ is smooth, its relative dimension is a Zariski locally constant function on $U_\alpha$ that we denote by $d_\alpha$. In other words, $d_\alpha$ simply consists of integers for each connected component of $U_\alpha$.  To state the next results, we will use the following notation:
\begin{equation}\label{eq:non_conn_twist}
\un_B(d_\alpha)[2d_\alpha]
 =\bigoplus_{x \in U_\alpha^{(0)}} \un_B(d_\alpha(x))[2d_\alpha(x)].
\end{equation}

\begin{num}
Let us first recall the computation obtained in \cite[8.4.3]{MNP}, for Chow motives. We work over a quasi-projective base $B$ over a perfect field $k$. Taking into account the comparison result of Jin \cite[Th. 3.17]{Jin1}, the category of relative Chow motives defined by Corti-Hanamura corresponds to the weight $0$ part of the triangulated category of rational mixed motives $\DM(B,\QQ)$.

Given a projective morphism $f:X \rightarrow B$ such that $X$ is smooth over $k$, one defines the Chow motive (\emph{i.e.} cohomological motive) of $X/B$ as:
\[
h_B(X)=f_*(\un_X).
\]
With this notation, \cite[Prop. 8.4.3]{MNP} can be stated as follows.
\end{num}

\begin{prop}\label{prop:NMP}
Assume $f:X \rightarrow B$ is a projective morphism of $k$-varieties, with $X$ smooth and admitting a flat relative cellular structure.  There is a canonical isomorphism (using notation \eqref{eq:non_conn_twist}) of the form:
\[
h_B(X)=\bigoplus_\alpha \un_B(-d_\alpha)[-2d_\alpha].
\]
\end{prop}

Note \cite[8.4.3]{MNP} is stated for $k=\mathbb C$ but this assumption is not used in the proof, which works over an arbitrary base field $k$.  We appeal to the assumption $k$ perfect implicitly via the comparison result (i.e., \cite{Jin1}).

\begin{num}
We can actually extend the previous computation to $\AA^1$-homotopy, if one restricts to \emph{oriented spectra}. Let us adopt the following definition.  We will say that the abstract triangulated motivic category $\T$ is \emph{oriented}
 if:
\begin{itemize}
\item There exists a premotivic adjunction $\varphi^*:\SH \leftrightarrows \T:\varphi_*$;
\item The ring spectrum $\varphi_*(\un_S)$ is oriented.
\end{itemize}
These assumptions in place, we will use the theory of Borel-Moore objects as described in \cite[section 1.3]{BD1}, together with the theory of fundamental classes as developed in \cite{Deg16} (see also \cite{DJK} for a more general account).

For the convenience of the reader, we recall the definitions and properties we use.  Given any separated $B$-scheme of finite type $\pi:Y \rightarrow B$, we define its \emph{Borel-Moore $\T$-motive}\footnote{Another appropriate terminology would be cohomological $\T$-motive with compact support.} by the formula:
\[
\Mbm(Y/B):=\pi_!(\un_Y).
\]
We will need the following properties:
\begin{enumerate}
\item[(P1)] According to the localization property of triangulated motivic categories,
 given any closed immersion $i:Z \rightarrow X$ of $B$-schemes with complementary
 open immersion $j:U \rightarrow X$, one obtains a distinguished triangle:
\[
\Mbm(U/B) \xrightarrow{j_*} \Mbm(X/B)
 \xrightarrow{i^*} \Mbm(Z/B) \rightarrow \Mbm(U/B)[1].
\]
\item[(P2)] Using the main construction of \cite[Th. 2.5.3]{Deg16}, one associates
 to any quasi-projective morphism $f:Y \rightarrow X$ of relative dimension $d$,
 a morphism (also called a fundamental class in \emph{loc. cit.}):
\[
f_!:\Mbm(Y/X)=f_!(\un_Y) \rightarrow \un_X(-d)[-2d].
\]
\item[(P3)] If in addition to the assumptions of the previous point,
 one suppose that $f$ is smooth, then the following diagram is commutative:
\[
\xymatrix@R=8pt@C=50pt{
f_!(\un_Y)\ar_{\mathfrak p_f}^\sim[dd]\ar^{f_!}[rd] & \\
 & \un_X(-d)[-2d] \\
f_\sharp(\un_Y)(-d)[-2d]\ar_{\quad f_*(-d)[-2d]}[ru] &
}
\]
where $\mathfrak p_f$ is the purity isomorphism associated with the smooth
 morphism $f$ in the six functors formalism,
 $f_\sharp(\un_Y)$ is the $\T$-motive represented by $Y/B$,
 and $f_*$ refers to the classical functoriality.
 Note in particular that $f_*$ is the image of the canonical
 map $\Sigma^\infty Y_+ \rightarrow \Sigma^\infty B_+$ of spectra over $B$
 under the map $\varphi^*:\SH(B) \rightarrow \T(B)$.
\item[(P4)] If under the assumptions of point (2),
 $X$ is separated of finite type over a scheme $B$, with structural morphism $p$,
 we obtain by applying $p_!$ a Gysin morphism:
\[
f_!:\Mbm(Y/B) \rightarrow \Mbm(X/B)(-d)[-2d].
\]
It follows from the compatibility of fundamental classes with composition
 \cite[2.5.3]{Deg16} that these Gysin morphism are compatible with composition.
\end{enumerate}
\end{num}
\begin{prop}\label{prop:cellular_decomposition}
We consider a morphism $f:X \rightarrow B$ with an lci cellular structure,
 and $B$ an arbitrary scheme.

Then, using notation \eqref{eq:non_conn_twist},
 there exists a canonical isomorphism in $\MGLmod(B)$:
$$
f_!(\un_X)=\bigoplus_\alpha \un_B(-d_\alpha)[-2d_\alpha]
$$
\end{prop}
\begin{proof}
We look at the localization triangle of property (P1)
 associated with the closed immersion
 $i_\alpha:X_{\alpha}:X_{\alpha-1} \rightarrow X_{\alpha}$
 with complementary open immersion $j_\alpha:U_\alpha \rightarrow X_\alpha$:
$$
\Mbm(U_\alpha/B) \xrightarrow{j_{\alpha*}} \Mbm(X_\alpha/B)
 \xrightarrow{i^*_\alpha} \Mbm(X_{\alpha-1}/B) \rightarrow \Mbm(U_\alpha/B)[1].
$$
Next we apply property (P4) to get a commutative diagram:
$$
\xymatrix@R=22pt@C=-5pt{
\Mbm(U_\alpha/B)\ar_/-5pt/{p_{\alpha*}}[rd]\ar^{j_{\alpha*}}[rr]
 && \Mbm(X_\alpha/B)\ar^/-8pt/{f_{\alpha*}}[ld] \\
& \Mbm(B/B)(-d_\alpha)[-2d_\alpha] &
}
$$
Point (P3) applied to the smooth morphism $p_\alpha$,
 which by assumption is a stable weak $\AA^1$-equivalence,
 implies the map $p_{\alpha*}$ is an isomorphism.
 In particular, $j_{\alpha*}$ is a split monomorphism
 (with spliting $p_{\alpha*}^{-1}f_{\alpha*}$) and one gets an isomorphism
 $\Mbm(U_\alpha/B)=\un_B(-d_\alpha)[-2d_\alpha]$.

This completes the proof.
\end{proof}

\begin{cor}
Assume that $f:X \rightarrow B$ is a projective morphism with lci cellular structure.
\begin{enumerate}
\item Then one gets an isomorphism of the following \emph{cohomological $\T$-motives}:
$$
f_*(\un_X)=\bigoplus_\alpha \un_B(-d_\alpha)[-2d_\alpha].
$$
In particular, one gets an isomorphism of $\T$-cohomology:
$$
H^{**}(X,\un_X) \simeq \bigoplus_\alpha H^{*-2d_\alpha,*-d_\alpha}(B,\un_B)
$$
\item For any $\T$-spectrum $\E$ over $B$, there exists a canonical isomorphism:
$$
f_*f^*(\E) \simeq \bigoplus_\alpha \un_B(-d_\alpha)[-2d_\alpha].
$$
\end{enumerate}
\end{cor}
Only the second point needs a proof.
 One simply uses the preceding proposition and
 the projection formula from the six functors formalism:
$$
f_*f^*(\E) \simeq f_!(\un_X \otimes f^*(\E)) \simeq f_!(\un_X) \otimes \E.
$$

\begin{ex}
\begin{enumerate}
\item In the motivic case, one gets back the result of Proposition \ref{prop:NMP}.
 In fact, with rational coefficients, we have shown that one can get rid of a base field,
 if one replaces the flatness by lci. (See also next Remark.)
\item We can apply the preceding result to classical oriented ring spectra,
 such as $\MGL$ and $\KGL$, by using the theory of modules over ring spectra
 (see \cite[\textsection 7]{CD3}). The corollary extends a previously known
 result in \cite{NZ}, from characteristic $0$ to the absolute case
 (without the need of a base field).
\end{enumerate}
\end{ex}

\begin{rem}
\begin{enumerate}
\item The above corollary applies to strict $\MGL$-modules $M$ over $B$.
 That is the $\MGL$-module structure on $M$ must be defined at
 the model category level. This assumption is too strong in many situations.
One can avoid it, assuming only that $M$ is a spectrum in $\SH(B)$
 with a module structure over $\MGL_B$.
 Indeed, it is possible to use the proof of Proposition \ref{prop:cellular_decomposition}
 after replacing $f_!(\un_X)=f_!f^*(\un_X)$ with the spectrum $f_!f^*(M)$.
 One uses the theory of fundamental classes as developed in \cite{DJK} and the fact
 that one has Thom isomorphisms for $M$, using the $\MGL_B$-module structure.

We leave the details to the reader.
\item The orientation assumptions in the preceding proposition and corollary
 is essential. Indeed, recall that the Chow-Witt ring
 does not satisfy the projective bundle formula though it is representable in $\SH$
 by the Milnor-Witt homotopy module.

Nevertheless, one can extend the validity of the above result by considering weaker
 orientability conditions. For example,
 we can consider a symplectically oriented (ring) spectrum $M$
 (see \cite{PW1}),
 provided that we give for each index $\alpha$ a symplectic structure on
 the tangent bundle of $U_\alpha/B$.
\item At least in the rational motivic case
 (which is equivalent to the rational orientable case), it is possible in principle
 to generalize both the flat and lci case. Indeed, it is visible in the proof that we only need
 a good theory of fundamental classes for the projections $X_\alpha \rightarrow B$.
 The recent work \cite{JinGth} opens the way to define these fundamental classes for
 arbitrary \emph{finite tor-dimension morphisms}, which contains both flat and lci cases.
 The main tool to do that is to extend Jin's work to a suitable representability theorem
 for relative K-theory (as defined in \cite[IV, 3.3]{SGA6}).
\end{enumerate}
\end{rem}

The relative cellular space provides us with a situation where the homotopy Leray spectral sequence is particularly simple.
\begin{prop}
Let $f:X \rightarrow B$ be a proper morphism with an lci relative cellular structure.
 We consider a homotopy module $M$ over $B$ satisfying one of the following conditions:
\begin{itemize}
\item \textit{Homotopical case}.-- $M$ is an oriented homotopy module in $\SH(B)$.
\item \textit{Motivic case}.-- $M$ is a homotopy module in $\DM(B,R)$.
\end{itemize}
We consider the homotopy module $M_X=f^!M$. Then there exists an isomorphism:
$$
f_*(M_X) \simeq \sum_{i=0}^d H^i_\delta(f_*M_X)[-i]
$$
where $d$ is the maximum of the dimension of the fibers of $f$.

Moreover, the homotopy Leray spectral sequence:
$$
E_2^{p,q}=A^p(B,H^q_\delta(f_*M_X)) \Rightarrow A^{p+q}(X,M_X)
$$
degenerates at $E_2$, and the abutting filtration splits giving an isomorphism:
$$
A^{n}(X,M_X) \simeq \bigoplus_{p=0}^n A^p(B,H^{n-p}_\delta(f_*M_X)).
$$
\end{prop} 

\bibliographystyle{amsalpha}
\bibliography{ssps}
\Addresses
\end{document}